\documentclass{amsart}[10pt,fullpage]

\usepackage{amsmath,amssymb,amsthm}
\usepackage[numbers,sort]{natbib}
\usepackage{multirow}
\usepackage{xcolor,mathrsfs,graphicx}
\usepackage{tikz}
\usetikzlibrary{arrows}

\newtheorem{thm}{Theorem}[section]
\newtheorem{lem}[thm]{Lemma}
\newtheorem{prop}[thm]{Proposition}
\newtheorem{cor}[thm]{Corollary}

\theoremstyle{definition}

\newtheorem{defn}[thm]{Definition}
\newtheorem{rem}[thm]{Remark}
\newtheorem*{ack}{Acknowledgments}

\numberwithin{equation}{section}
\input{xy}
\xyoption{all}

\newcommand{\comment}[1]{}

\newcommand{\ve}{\varepsilon}

\newcommand{\bd}{\partial}

\newcommand{\ints}{\mathbb Z}
\newcommand{\rls}{\mathbb R}

\newcommand{\rat}{\mathbb Q}
\newcommand{\bb}[1]{\mathbb{#1}}
\newcommand{\cl}[1]{\mathcal{#1}}
\newcommand{\scr}[1]{\mathscr{#1}}

\newcommand{\wt}[1]{\widetilde{#1}}

\newcommand{\al}[1]{\begin{align*}#1\end{align*}}
\newcommand{\en}[1]{\begin{enumerate}#1\end{enumerate}}

\newcommand{\set}[1]{\left\{#1\right\}}
\newcommand{\setn}[2]{\left\{#1\mid #2\right\}}

\newcommand{\inp}[2]{\left<#1,#2\right>}
\newcommand{\abs}[1]{\left\vert#1\right\vert}

\newcommand{\tb}{\text{tb}}
\newcommand{\rot}{\text{rot}}
\newcommand{\slk}{\text{sl}}

\begin{document}

\title[]{Bennequin type inequalities in lens spaces}
\author[]{Christopher R. Cornwell}
\address[]{Department of Mathematics, Michigan State University, East
Lansing, MI 48824}
\email[]{cornwell@math.msu.edu}

\begin{abstract}
We give criteria for an invariant of lens space links to bound the maximal self-linking number in certain tight contact lens spaces. Our result generalizes that given by Ng \cite{Ng} for links in $S^3$ with the standard tight contact structure. As a corollary we extend the Franks-Williams-Morton inequality to the setting of lens spaces.
\end{abstract}

\maketitle

\section{Introduction}
\label{sec:intro}

Of intrinsic interest to the study of Legendrian and transverse knots are the Thurston-Bennequin number $\tb(K)$ and rotation number $\rot(K)$ in the Legendrian setting, and the self-linking number $\slk(K)$ in the transverse setting. Along with the knot type, they are known as the ``classical invariants'' of Legendrian and transverse knots \cite{B}. Much effort has gone into finding upper bounds for these classical invariants in $(S^3,\xi_{st})$, where $\xi_{st}$ is the standard tight contact structure on the 3-sphere (e.g. \cite{Bnn},\cite{FW,M},\cite{Rud1},\cite{Rud2},\cite{Pla1},\cite{Pla2,Shum},\cite{Ng1},\cite{Wu}). Much of this work benefits from the fact that the classical invariants of Legendrian/transverse knots in $(S^3,\xi_{st})$ can be computed easily from a front projection.

Less is known about these invariants of Legendrian knots in other contact manifolds. A theorem of Eliashberg \cite{E} generalizes the Bennequin inequality for null-homologous Legendrian knots in any 3-manifold with a tight contact structure:

\begin{thm}[Eliashberg-Bennequin inequality] Let $\xi$ be a tight contact structure on a 3-manifold, $Y$. If $K$ is a null-homologous knot in $Y$ and $F$ is a Seifert surface for $K$, then
	\begin{equation}\emph{tb}(K_l)+\abs{\emph{rot}_F(K_l)}\le 2g(F)-1
	\label{eqn:EBineq}
	\end{equation}
for any $K_l$, a Legendrian representative of $K$.
\end{thm}

This bound can be improved in some settings. Lisca and Matic improved the bound in the case that the contact structure is Stein fillable \cite{LM}, and this improvement was extended to the setting of a tight contact structure with non-vanishing Seiberg-Witten contact invariant by Mrowka and Rollin \cite{MR}. An analogous theorem was proved by Wu \cite{Wu2} for the Ozsv\'ath-Szab\'o contact invariant. 

These improvements involved replacing the Seifert genus in the Eliashberg-Bennequin inequality with the genus of a surface which is properly embedded in a 4-manifold bounded by $Y$. As such bounds involve the negative Euler characteristic of a surface with boundary $K$, they must be no less than -1. In \cite{H}, Hedden introduced an integer $\tau_{\xi}(K)$ that is defined via the filtration on knot Floer homology associated to $(Y,[F],K)$, where $[F]$ is the homology class of a Seifert surface for $K$. He showed that in the case that the Ozsv\'ath-Szab\'o contact invariant is non-zero, the right side of (\ref{eqn:EBineq}) can be replaced by $2\tau_{\xi}(K)-1$. With such a bound he showed that for any contact manifold with non-zero contact invariant, there exist prime Legendrian knots with arbitrarily negative classical invariants.

In another direction, one could consider rationally null-homologous knots in a contact manifold $(Y,\xi)$. In such a setting there is a notion of rational Seifert surface and corresponding classical invariants $\tb_\rat$, $\rot_\rat$, and $\slk_\rat$ (see Definition \ref{defn:classicalInvts} below). Baker and Etnyre \cite{BE} extend the Eliashberg-Bennequin inequality to this setting:

\begin{thm}
Let $(Y,\xi)$ be a contact 3-manifold with $\xi$ a tight contact structure. Let $K$ be a knot in $Y$ with order $r>0$ in homology and let $\Sigma$ be a rational Seifert surface for $K$. Then for $K_t$, a transverse representative of $K$,
	\[\slk_\rat(K_t)\le-\frac1r\chi(\Sigma).\]
Moreover, if $K_l$ is a Legendrian representative of $K$ then
	\[\tb_\rat(K_l)+\abs{\rot_\rat(K_l)}\le-\frac1r\chi(\Sigma).\]
\end{thm}
	
There is an inequality found by Franks and Williams \cite{FW}, and independently by Morton \cite{M}, that relates the index and algebraic crossing number of a braid to a degree of the HOMFLY polynomial of its closure. Later, using the work of Bennequin, Fuchs and Tabachnikov reinterpreted the result in terms of the self-linking number of a transverse knot in $(S^3,\xi_{st})$ \cite{FT}. This inequality has come to be known as the Franks-Williams-Morton inequality.

More precisely, we describe the Franks-Williams-Morton (FWM) inequality as follows. The HOMFLY polynomial $J(K)$ is a polynomial invariant of links in the variables $v,z$ such that if $U\subset S^3$ is the unknot then $J(U)=1$, and $J$ satisfies
	\begin{equation}
	v^{-1}J(K_+)-vJ(K_-)=zJ(K_0),
	\label{eqn:skeinRel'n}
	\end{equation}
where $K_+, K_-,$ and $K_0$ differ only in a small neighborhood as below.

	
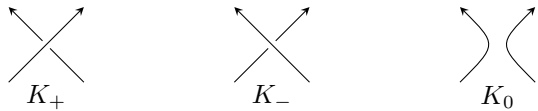
\[\begin{tikzpicture}[>=stealth]
		\draw[->]
			(1,0)--(0.55,0.45)
			(0.45,0.55)--(0,1);
		\draw[->]
			(0.5,-0.2) node {$K_+$}
			(0,0)--(1,1);
		\draw[->]
			(3,0)--(3.45,0.45)
			(3.55,0.55)--(4,1);
		\draw[->]
			(3.5,-0.2) node {$K_-$}
			(4,0)--(3,1);
		\draw[->]
			(6,0) .. controls (6.5,0.5) .. (6,1);
		\draw[->]
			(6.5,-0.2) node {$K_0$}
			(7,0) .. controls (6.5,0.5) .. (7,1);
	\end{tikzpicture}\]

\begin{thm}[Franks-Williams-Morton inequality] Let $e(K)$ denote the minimum degree of $v$ in $J(K)$. Then for any transverse representative $K_t$ of $K$,
	\[\slk(K_t)\le e(K)-1.\]
Moreover, if $K_l$ is a Legendrian representative of $K$ then
	\[\tb(K_l)+\abs{\rot(K_l)}\le e(K)-1.\]
\label{thm:FWMineq}
\end{thm}

In this paper we give criteria for a $\rat$-valued invariant of links in $L(p,q)$ to bound the classical invariants in $(L(p,q),\xi_{UT})$, where $\xi_{UT}$ is a universally tight contact structure on $L(p,q)$ defined by the pushforward of $\xi_{st}$. Our result is a lens space analogue of a theorem of Lenny Ng \cite{Ng}. Necessary to the theorem is the definition of a collection of links in $L(p,q)$, called \emph{trivial links}, which has one representative in each homotopy class of links. These trivial links were defined explicitly in \cite{C} via grid diagrams that correspond to toroidal front projections in the contact lens space mentioned above (see \cite{BG}). Moreover, we note that in \cite{C} a skein theory was developed, producing a finite length skein tree for any link in $L(p,q)$ with the trivial links as leaves in the tree. 

To a given trivial link $\tau$ let $T(\tau)$ be the transverse representative of $\tau$ defined in Remark \ref{rem:T(K)}. Our main theorem is as follows.

\begin{thm} Let $i$ be a $\rat$-valued invariant of oriented links in $L(p,q)$ such that
	\al{
	&i(L_+)+1\le\max\left(i(L_-)-1,i(L_0)\right)\\
	and\\
	&i(L_-)-1\le\max\left(i(L_+)+1,i(L_0)\right),
	}
where $L_+, L_-,$ and $L_0$ are oriented links that differ as in the skein relation. If $\emph{sl}_\rat(T(\tau))\le-i(\tau)$ for every trivial link $\tau$ in $L(p,q)$, then
		\[\overline{\emph{sl}}_\rat(L)\le-i(L)\]
	 for every link $L$ in $L(p,q)$. Here $\overline{\emph{sl}}_\rat(L)$ is the maximum rational self-linking number among transverse links in $(L(p,q),\xi_{UT})$ that are isotopic to $L$.
\label{mainThm2}
\end{thm}

In order to prove the theorem we provide explicit formulae to calculate the invariants $\tb_\rat, \rot_\rat,$ and $\slk_\rat$ from a projection of the link to a Heegaard torus. These formulae are in the spirit of those used to compute the classical invariants in $(S^3,\xi_{st})$ from a front projection (see \cite{B}). Moreover, combining the formula for $\tb_\rat$ with a result of Baker and Grigsby \cite{BG}, we find a very short proof of a well-known result of Fintushel and Stern \cite{FS}, that if integral surgery on a knot in $L(p,q)$ yields $S^3$, then $\pm q$ is a quadratic residue mod $p$.

In \cite{KL}, Kalfagianni and Lin gave a power series invariant of oriented links in a large family of rational homology 3-spheres which satisfies the HOMFLY skein relation (\ref{eqn:skeinRel'n}). The work of the author \cite{C} shows this power series to coverge in a lens space $L(p,q)$ to a Laurent polynomial, providing a HOMFLY polynomial $J_{p,q}$ for links in $L(p,q)$. The polynomial $J_{p,q}(K)$ is a two-variable polynomial in variables $a$ and $z$, and its definition depends on a normalization on the set of trivial links. As a corollary of Theorem \ref{mainThm2} we are able to extend the FWM inequality to the setting of links in $(L(p,q),\xi_{UT})$. A second corollary then tells us the maximal self-linking number of any trivial link and the maximal Thurston-Bennequin number of a trivial knot, using the formulae derived in Section \ref{sec:gridProjForms}.

\begin{cor}
Let $J_{p,q}$ denote the HOMFLY polynomial invariant in $L(p,q)$, normalized so that if $\tau$ is a trivial link with no nullhomotopic components, or is the unknot, then $J_{p,q}(\tau)=a^{p\cdot\emph{sl}_\rat(T(\tau))+1}$. Given an oriented link $L$ with transverse representative $L_t$ in $(L(p,q),\xi_{UT})$, set $e(L)$ to be the minimum degree in $a$ of $J_{p,q}(L)$. Then
	\[\emph{sl}_\rat(L_t)\le\frac{e(L)-1}p.\]
\label{BennBd}
\end{cor}

\begin{cor}
If $\tau$ is a trivial link in $L(p,q)$, then $T(\tau)$ has maximal self-linking number among all transverse representatives of $\tau$. If $\tau$ is a trivial knot, then the Legendrian knot associated to its grid number one diagram has maximal Thurston-Bennequin number.
\label{TrivmaxSL}
\end{cor}

The paper is organized as follows: In Section \ref{sec:Prelim} we review the construction of toroidal grid diagrams in a lens space made originally in \cite{BG}, \cite{BGH}. We also review results from \cite{C} that provide the lens space HOMFLY polynomial and develop a skein theory on grid diagrams. Then we review constructions of \cite{BE},\cite{BG} that extend classical invariants of Legendrian and transverse knots to the rationally null-homologous setting. In Section \ref{sec:gridProjForms} we give formulas for the Thurston-Bennequin, rotation, and self-linking numbers. In Section \ref{sec:BennBd} we prove our main results, Theorem \ref{mainThm2} and Corollaries \ref{BennBd} and \ref{TrivmaxSL}. Finally Section \ref{sec:comput} gives a sequence of Legendrian knots and links in $(L(5,1),\xi_{UT})$ on which the FWM inequality is sharp and arbitrarily negative.

\begin{ack}The author would like to thank Eli Grigsby, Matt Hedden, and Lenny Ng for their help and input. He thanks his advisor Effie Kalfagianni for introducing him to these problems, for her expertise, and for many helpful discussions. He also thanks the referee for very useful suggestions. This research was supported in part by NSF--RTG grant DMS-0353717, NSF grant DMS-0805942, and by a Herbert T. Graham scholarship.
\end{ack}

\section{Preliminaries}
\label{sec:Prelim}
\subsection{Legendrian links and grid diagrams in $L(p,q)$}
\label{subsec:lensGrids}

Let $L(p,q)$ be the lens space obtained as a quotient of $S^3\subset\bb C^2$ by the equivalence relation $(u_1,u_2)\sim(\omega_pu_1,\omega_p^qu_2)$, where $\omega_p=e^{\frac{2\pi i}p}$. Let $\pi:S^3\to L(p,q)$ be the quotient map.

Represent points $(u_1,u_2)$ of $S^3$ in polar coordinates, letting $u_i=(r_i, \theta_i)$. The kernel $\xi_{st}$ of the 1-form $\alpha=r_1^2d\theta_1+r_2^2d\theta_2$ is the unique (up to orientation) tight contact structure on $S^3$~\cite{Ge}.
The 1-form $\alpha$ is constant along any torus in $S^3$ determined by a fixed $r_1$. Since such a torus is fixed (not pointwise) under the action $(u_1,u_2)\mapsto (\omega_pu_1,\omega_p^qu_2)$, the pushforward $\xi_{UT}=\pi_*(\xi_{st})$ is a well-defined tight contact structure on $L(p,q)$.

The points of $L(p,q)$ can be identified with points in a fundamental domain of the cyclic action on $S^3$. Thus, since $r_2$ is determined by $r_1$ in $S^3$, we can describe $L(p,q)$ by
	\[L(p,q)=\setn{(r_1,\theta_1,\theta_2)}{r_1\in[0,1],\ \theta_1\in\left[0,2\pi\right),\ \theta_2\in\left[0,\frac{2\pi}p\right)}.\]

Analogous to the correspondence between planar grid diagrams and Legendrian links in $(S^3,\xi_{st})$ (\cite{NT},\cite{Rud3}), we can define toroidal grid diagrams in $L(p,q)$ to get a correspondence between grid diagrams and Legendrian links in $(L(p,q),\xi_{UT})$. To be precise we define grid diagrams in $L(p,q)$ as follows.

\begin{defn} A \emph{grid diagram} $D$ with grid number $n$ in $L(p,q)$ is a set of data $(T,\vec\alpha, \vec\beta, \vec{\bb O},\vec{\bb X})$, where:
	\begin{itemize}
		\item $T$ is the oriented torus obtained via the quotient of $\rls^2$ by the $\ints^2$ lattice generated by $(1,0)$ and $(0,1)$.
		\item $\vec\alpha=\set{\alpha_0,\ldots,\alpha_{n-1}}$, with $\alpha_i$ the image of the line $y=\frac in$ in $T$. Call the $n$ annular components of $T-\vec\alpha$ the \emph{rows} of the grid diagram.
		\item $\vec\beta=\set{\beta_0,\ldots,\beta_{n-1}}$, with $\beta_i$ the image of the line $y=-\frac p{q}(x-\frac i{pn})$ in $T$. Call the $n$ annular components of $T-\vec\beta$ the \emph{columns} of the grid diagram.
		\item $\vec{\bb O}=\set{O_0, \ldots, O_{n-1}}$ is a set of $n$ points in $T-\vec\alpha-\vec\beta$ such that no two $O_i$'s lie in the same row or column.
		\item $\vec{\bb X}=\set{X_0, \ldots,X_{n-1}}$ is a set of $n$ points in $T-\vec\alpha-\vec\beta$ such that no two $X_i$'s lie in the same row or column.
	\end{itemize}
	
	The components of $T-\vec\alpha-\vec\beta$ are called the \emph{fundamental parallelograms} of $D$ and the points $\vec{\bb O}\cup\vec{\bb X}$ are called the \emph{markings} of $D$. Two grid diagrams with corresponding tori $T_1, T_2$ are considered equivalent if there exists an orientation-preserving diffeomorphism $T_1\to T_2$ respecting the markings (up to cyclic permutation of their labels).
\label{gridDefn}
\end{defn}

Such a grid diagram has ``slanted'' $\beta$ curves. For considerations of both convenience and aesthetics, we alter the fundamental domain of $T$ and ``straighten'' our pictures so that the $\beta$ curves are vertical. Figure \ref{fig:Strait} shows how this ``straightening'' is accomplished.

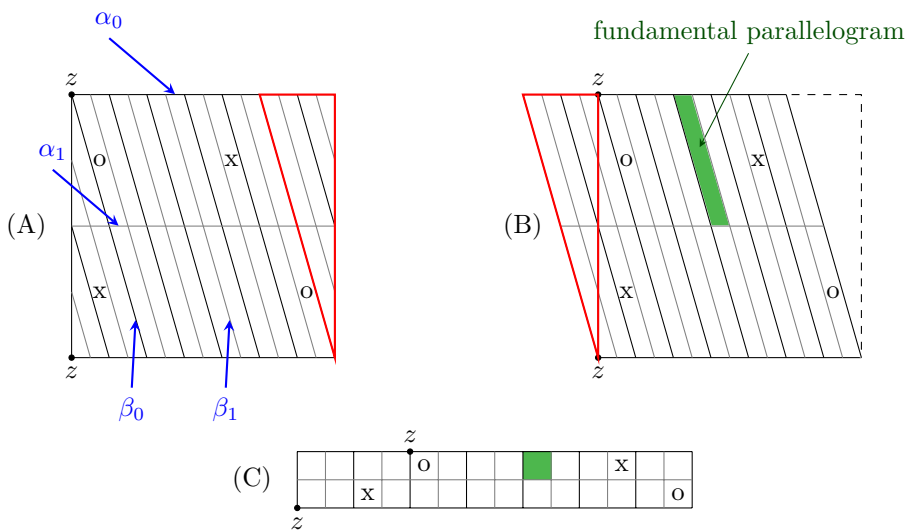
\begin{figure}[ht]
	\[\begin{tikzpicture}[>=stealth]
	\fill[green!60!black,fill opacity=0.70]
		(8.01,3.49) -- (8.51,1.76) -- (8.74,1.76) -- (8.24,3.49)--cycle;
	\fill[green!60!black, fill opacity=0.70]
		(6,-1.625) -- (6.375,-1.625) -- (6.375,-1.25) -- (6,-1.25) --cycle;
	\draw[green!30!black,->]
		(9,4.05) -- node [above, at start] {fundamental parallelogram} (8.35,2.8);
		
		\draw[step=3.5]	(0,0) grid (3.5,3.5);
	\foreach \x in {2,3,4,5,6,7}
		\draw (0.5*\x,0)--(0.5*\x-1,3.5);
	\foreach \x in {2,3,4,5,6}
		\draw[gray,thin] (0.5*\x+0.25,0)--(0.5*\x-0.75,3.5);
	\foreach \r in {0,180}
		\draw[rotate around = {\r:(1.75,1.75)}]
			(0.5,0) -- (0,1.75);
	\foreach \x in {0,1}
	\foreach \r in {0,180}
		\draw[gray,thin,rotate around={\r:(1.75,1.75)}]
			(0.25+0.5*\x,0) -- (0,0.875+1.75*\x);
	\draw[gray,thin]
		(0,1.75) -- (3.5,1.75);
	\draw
		(0.375,0.875) node {x}
		(0.375,2.625) node {o}
		(2.125,2.625) node {x}
		(3.125,0.875) node {o};

	\draw[step=3.5,cm={1,0,0,1,(7,0)}]
		(0,0) grid (3.5,3.5);
	\draw[white,dashed,very thick]
		(10.5,0)--(10.5,3.5)--(9.5,3.5);
	\draw
		(6,3.5)--(7,3.5);
	\foreach \x in {0,1,2,3,4,5,6,7}
		\draw[cm={1,0,0,1,(7,0)}] (0.5*\x,0)--(0.5*\x-1,3.5);
	\foreach \x in {0,1,2,3,4,5,6}
		\draw[gray,thin,cm={1,0,0,1,(7,0)}] (0.5*\x+0.25,0)--(0.5*\x-0.75,3.5);
	\draw[gray,thin,cm={1,0,0,1,(7,0)}]
		(-0.5,1.75) -- (3,1.75);
	\draw[cm={1,0,0,1,(7,0)}]
		(0.375,0.875) node {x}
		(0.375,2.625) node {o}
		(2.125,2.625) node {x}
		(3.125,0.875) node {o};
	
	\draw[step=0.75,cm={1,0,0,1,(3,-2)}]
		(0,0) grid (5.25,0.75);
	\draw[gray,thin,step=0.75,cm={1,0,0,1,(3.375,-1.625)}]
		(-0.375,-0.375) grid (4.875,0.375);
	\draw[cm={1,0,0,1,(3,-2)}]
		(0.9375,0.1875) node {x}
		(1.6875,0.5625) node {o}
		(5.0625,0.1875) node {o}
		(4.3125,0.5625) node {x};
	\foreach \x in {0,7}
	\foreach \y in {0,3.5}
	\filldraw
		(\x,\y) circle (1pt);
	\foreach \x in {0,7}
	\draw
		(\x,0) node [below] {$z$}
		(\x,3.5) node [above] {$z$};
	\filldraw
		(3,-2) circle (1pt)
		(4.5,-1.25) circle (1pt);
	\draw
		(3,-2) node [below] {$z$}
		(4.5,-1.25) node [above] {$z$};
	\draw[thick]
		(-0.6,1.75) node {(A)}
		(6,1.75) node {(B)}
		(2.4,-1.625) node {(C)};
	
	\foreach \t in {0,3.5}
	\draw[thick,red]
		(3.5+\t,0) -- (2.5+\t,3.5) -- (3.5+\t,3.5)--cycle;
	\draw[thick, blue,->]
		(0.5,4.25) -- node [above,at start] {$\alpha_0$} (1.375,3.505);
	\draw[thick,blue,->]
		(-0.25,2.5) -- node [above,at start] {$\alpha_1$} (0.625,1.755);
	\draw[thick, blue,->]
		(0.8,-0.4) -- node [below,at start] {$\beta_0$} (0.85,0.5);
	\draw[thick,blue,->]
		(2.05,-0.4) -- node [below,at start] {$\beta_1$} (2.1,0.5);
		
	\end{tikzpicture}\]
\caption{(A) shows a grid diagram (with grid number 2) in $L(7,2)$ on a fundamental domain of $T$. In (B) we alter the fundamental domain. (C) is the ``straightening'' of (B).}
\label{fig:Strait}
\end{figure}

The correspondence between grid diagrams in $L(p,q)$ and Legendrian links in $(L(p,q),\xi_{UT})$ was fully developed in \cite{BG}. We note that in the association of a Legendrian link to a grid diagram in $L(p,q)$ the author used the opposite convention (in \cite{C}) as that adopted in other places in the literature \cite{OST,OS,BG,BGH}. The convention used in other places in the literature was adopted to fit conventions coming from knot Floer homology. However, for the purposes of this paper it is more clear to use the approach presented below (as in \cite{C}) as there is no reference to Floer homology theories.

A link $K\subset L(p,q)$ is associated to a grid diagram $D$ in $L(p,q)$ in the following manner. Let $\Sigma$ be the torus in $L(p,q)$ of constant radius $r_1=1/\sqrt{2}$ which splits $L(p,q)$ into two solid tori $V^\alpha$ and $V^\beta$. Identify $T$ with $-\Sigma$ such that the $\alpha$-curves of $D$ are negatively-oriented meridians of $V^\alpha$ and the $\beta$-curves are meridians of $V^\beta$. Next connect each $X$ to the $O$ in its row by an ``horizontal'' oriented arc (from $X$ to $O$) that is embedded in $T$ and disjoint from $\vec\alpha$. Likewise, connect each $O$ to the $X$ in its column by a ``vertical'' oriented arc embedded in $T$ and disjoint from $\vec\beta$. The union of the $2n$ arcs makes a multicurve $\gamma$. Remove self-intersections of $\gamma$ by pushing the interiors of horizontal arcs up into $V^\alpha$ and the interiors of vertical arcs down into $V^\beta$. 

\begin{defn}
Let $K$ be a link associated to a grid diagram $D$ in $L(p,q)$ with grid number $n$. For some $0<m<n$, suppose $D'$ is a subcollection of $m$ rows and $m$ columns of $D$ such that the $2m$ markings contained in the rows of $D'$ are exactly the $2m$ markings contained in the columns of $D'$. Then $D'$ is a grid diagram for some sublink of $K$. If this sublink has one component then $D'$ is called a \emph{component of $D$}.
\label{defn:compDiag}
\end{defn}

\begin{rem}
No part of Definition \ref{gridDefn} prohibits a marking in $\bb X$ and a marking in $\bb O$ from being in the same fundamental parallelogram. To a grid diagram that has grid number one (and so, only one marking in $\bb X$ and one marking in $\bb O$) and its two markings in the same fundamental parallelogram, we associate a knot in $L(p,q)$ that is contained in a small ball neighborhood and bounds an embedded disk. 
\label{rem:unknot}
\end{rem}

\begin{rem} Except for the case described in Remark \ref{rem:unknot}, we assume that each marking of $D$ is the center point of the fundamental parallelogram that contains it.  Let the straightened fundamental domain of $T$ have normalized coordinates $\setn{(\theta_1,\theta_2)}{\theta_1\in[0,p],\ \theta_2\in[0,1]},$ so that each $O$ and $X$ sharing the same column have the same $\theta_1$-coordinate mod 1. 
\label{rpRem}
\end{rem}

Under the requirements of Remark \ref{rpRem}, the projection of $K$ to $-\Sigma$ (with vertical arcs crossing under horizontal arcs) is called a \emph{grid projection} associated to $D_K$ (the authors of \cite{BG} call this a rectilinear projection). Note that $K$ has an orientation given by construction and so the grid projection is also oriented. Figure \ref{gdexamFig} shows an example of a grid diagram with a corresponding grid projection.

A slight perturbation of a grid projection gives a (toroidal) front projection, which determines a Legendrian link in $(L(p,q),\xi_{UT})$. The cusps of the front projection correspond to lower-left and upper-right corners of the grid projection, and we call these corners the \emph{cusps} of a grid projection. 

\begin{figure}[ht]
	\[\begin{tikzpicture}[scale=0.115]
		\draw[step=21,thick]
			(0,0) grid (105,21);
		\draw[gray,thin,step=3]
			(0,0) grid (105,21);	
	\foreach \o in {(22.5,19.5),(31.5,1.5),(55.5,13.5),(58.5,10.5),(88.5,7.5),(91.5,4.5),(103.5,16.5)}
		\draw
			\o node {o};
	\foreach \x in {(25.5,16.5),(22.5,1.5),(40.5,10.5),(55.5,7.5),(58.5,4.5),(91.5,13.5),(94.5,19.5)}
		\draw
			\x node {x};
	\draw[blue,thick]
		(0,16.5) -- (25.5,16.5) -- (25.5,2)
		(25.5,1) -- (25.5,0)
		(31.5,0) -- (31.5,1.5) -- (22.5,1.5) -- (22.5,16)
		(22.5,17) -- (22.5,19.5) -- (94.5,19.5) -- (94.5,21)
		(40.5,0) -- (40.5,10.5) -- (58.5,10.5) -- (58.5,8)
		(58.5,7) -- (58.5,4.5) -- (91.5,4.5) -- (91.5,13.5) -- (55.5,13.5) -- (55.5,11)
		(55.5,10) -- (55.5,7.5) -- (88.5,7.5) -- (88.5,13)
		(88.5,14) -- (88.5,19)
		(88.5,20) -- (88.5,21)
		(103.5,21) -- (103.5,16.5) -- (105,16.5);
	\foreach \l in {0,1,2,3,4}
	\draw 
		(21*\l,0) node [below] {\l};
	\draw	(105,0) node [below] {0};	
	\foreach \l in {0,1,2}
	\draw
		(21*\l+63,21) node[above] {\l};
	\foreach \l in {2,3,4}
	\draw
		(21*\l-42,21) node[above] {\l};
	\filldraw
		(0,0) circle (16pt)
		(63,21) circle (16pt);
	
	\end{tikzpicture}\]
	\caption{A grid diagram for $L(5,3)$ with corresponding grid projection.}
	\label{gdexamFig}
\end{figure}

If $D$ has grid number $n$ then there are $2^{2n}$ different grid projections as there are two choices of vertical arc for each column, and two choices of horizontal arc for each row. In a given row (resp.\ column), the difference in choice of horizontal (resp.\ vertical) arc corresponds to a Legendrian isotopy across a meridional disk of $V^\alpha$ (resp.\ $V^\beta$). So the Legendrian isotopy class of the link is independent of the choice of grid projection. For more details we refer the reader to \cite{BG}.

Given a grid diagram in a lens space $L(p,q)$ (with the identification of $T$ to $-\Sigma$), the basis of vectors given by parallel translates of tangent vectors to $\{\alpha_0,\beta_0\}$ is coherently oriented with the global frame $\{d/d\theta_1,d/d\theta_2\}$.

In view of the correspondence above, Legendrian links in $(L(p,q),\xi_{UT})$ can be discussed via grid diagrams. There is a set of grid moves such that two grid diagrams correspond to the same Legendrian link if and only if there is a sequence of such grid moves taking one grid diagram to the other \cite{BG}. These moves come in two flavors: grid (de)stabilizations and commutations. 

\begin{figure}[ht]
	\begin{tikzpicture}[>=stealth]
	\draw[step=0.75,cm={1,0,0,1,(0,0)}]
		(0,0) grid (5.25,0.75);
	\draw
		(1.125,0.375) node {X}
		(4.875,0.375) node {O};
	\draw[->]
		(5.5,0.5) -- (6.5,0.5);
	\draw[step=0.75,cm={1,0,0,1,(6.7,0)}]
		(0,0) grid (5.25,0.75);
	\draw[gray,thin,step=0.75,cm={1,0,0,1,(7.075,0.375)}]
		(-0.375,-0.375) grid (4.875,0.375);
	\draw[cm={1,0,0,1,(6.7,0)}]
		(0.9375,0.1875) node {x}
		(1.3125,0.1875) node {o}
		(4.6875,0.5625) node {o}
		(1.3125,0.5625) node {x};
	\draw[step=0.75,cm={1,0,0,1,(0,-2)}]
		(0,0) grid (5.25,0.75);
	\draw[cm={1,0,0,1,(0,-2)}]
		(1.125,0.375) node {X}
		(4.875,0.375) node {O};
	\draw[->,cm={1,0,0,1,(0,-2)}]
		(5.5,0.5) -- (6.5,0.5);
	\draw[step=0.75,cm={1,0,0,1,(6.7,-2)}]
		(0,0) grid (5.25,0.75);
	\draw[gray,thin,step=0.75,cm={1,0,0,1,(7.075,-1.625)}]
		(-0.375,-0.375) grid (4.875,0.375);
	\draw[cm={1,0,0,1,(6.7,-2)}]
		(0.9375,0.1875) node {x}
		(1.3125+3.75,0.1875) node {o}
		(4.6875,0.5625) node {o}
		(1.3125+3.75,0.5625) node {x};

	\end{tikzpicture}
\caption{Stabilization of type X:NW (top) and of type O:SW (bottom)}
\label{fig:stabl}
\end{figure}
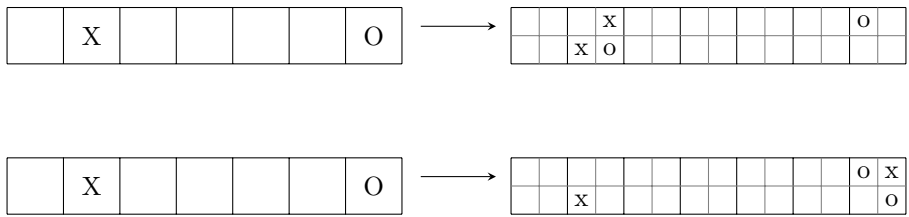

{\bf Grid Stabilizations and Destabilizations:} Grid stabilizations increase the grid number by one and should be thought of as adding a local kink to the knot. They are named with an X or O, depending on the type of marking at which stabilization occurs, and with NW,NE,SW, or SE, depending on the positioning of the new markings. Figure \ref{fig:stabl} shows an X:NW stabilization and an O:SW stabilization. Destabilizations are the inverse of a stabilization. Any (de)stabilization is a grid move that preserves the isotopy type. 

However, the correspondence between our grid diagrams and toroidal front projections is such that cusps correspond to upper-right and lower-left corners of a grid projection. Only (de)stabilizations of types NW and SE preserve Legendrian isotopy type.

\begin{figure}[ht]
	\[\begin{tikzpicture}[scale=0.15,>=stealth]
		\draw[step=8,thick]
			(0,0) grid (56,8);
		\draw[step=2,gray,thin]
			(0,0) grid (56,8);
	
	\draw[blue, thick]
		(5,5) -- (5,0)
		(21,8) -- (21,5.5) (21,4.5) -- (21,1) -- (7,1) -- (7,0)
		(23,8) -- (23,5.5) (23,4.5) -- (23,0)
		(39,8) -- (39,7) -- (56,7)
		(0,7) -- (3,7) -- (3,8)
		(43,0) -- (43,2.5) (43,3.5) -- (43,6.5) (43,7.5) -- (43,8)
		(27,0) -- (27,4.5) (27,5.5) -- (27,8)
		(11,0) -- (11,0.5) (11,1.5) -- (11,4.5) (11,5.5) -- (11,8)
		(51,0) -- (51,6.5) (51,7.5) -- (51,8)
		(35,0) -- (35,3) -- (49,3) -- (49,6.5) (49,7.5) -- (49,8)
		(33,0) -- (33,5) -- (5,5);
	
	\draw
		(3,7) node {x}
		(35,3) node {o}
		(5,5) node {x}
		(21,1) node {o};
	\draw
		(49,3) node {x}
		(33,5) node {o}
		(7,1) node {x}
		(39,7) node {o};
		
	\foreach \l in {0,1,2,3,4,5,6}
	\draw 
		(8*\l,0) node [below] {\l};
	\draw	(56,0) node [below] {0};	
	\foreach \l in {0,1,2,3,4,5}
	\draw
		(8*\l+16,8) node[above] {\l};
	\foreach \l in {5,6}
	\draw
		(8*\l-40,8) node[above] {\l};
	
	\foreach \c in {0,1,2,3,4,5,6}
		\draw[<->]
			(3+8*\c,9) -- (5+8*\c,9);
	
	\filldraw
		(0,0) circle (10pt)
		(16,8) circle (10pt);
		
	\draw[step=8,thick]
		(24,-16) grid (80,-8);
	\draw[step=2,gray,thin]
		(24,-16) grid (80,-8);	

	\draw[thick,blue,cm={1,0,0,1,(24,-16)}]
		(3,5) -- (3,0)
		(19,8) -- (19,5.5) (19,4.5) -- (19,1) -- (7,1) -- (7,0)
		(23,8) -- (23,5.5) (23,4.5) -- (23,0)
		(39,8) -- (39,7) -- (56,7)
		(0,7) -- (5,7) -- (5,8)
		(45,0) -- (45,2.5) (45,3.5) -- (45,6.5) (45,7.5) -- (45,8)
		(29,0) -- (29,4.5) (29,5.5) -- (29,8)
		(13,0) -- (13,0.5) (13,1.5) -- (13,4.5) (13,5.5) -- (13,8)
		(53,0) -- (53,6.5) (53,7.5) -- (53,8)
		(37,0) -- (37,3) -- (49,3) -- (49,6.5) (49,7.5) -- (49,8)
		(33,0) -- (33,5) -- (3,5);
	\draw[cm={1,0,0,1,(24,-16)}]
		(5,7) node {x}
		(37,3) node {o}
		(3,5) node {x}
		(19,1) node {o};
	\draw[cm={1,0,0,1,(24,-16)}]
		(49,3) node {x}
		(33,5) node {o}
		(7,1) node {x}
		(39,7) node {o};
		
	\foreach \l in {0,1,2,3,4,5,6}
	\draw [cm={1,0,0,1,(24,-16)}]
		(8*\l,0) node [below] {\l};
	\draw[cm={1,0,0,1,(24,-16)}]	(56,0) node [below] {0};	
	\foreach \l in {0,1,2,3,4,5}
	\draw[cm={1,0,0,1,(24,-16)}]
		(8*\l+16,8) node[above] {\l};
	\foreach \l in {5,6}
	\draw[cm={1,0,0,1,(24,-16)}]
		(8*\l-40,8) node[above] {\l};
	\draw[->,thick]
		(8,-5) .. controls (4,-12) and (21,-12) .. (23,-12);
	
	\filldraw
		(24,-16) circle (10pt)
		(40,-8) circle (10pt);
	
	\end{tikzpicture}\]
	\caption{A non-interleaving commutation in $L(7,2)$}
	\label{fig:commut}
\end{figure}
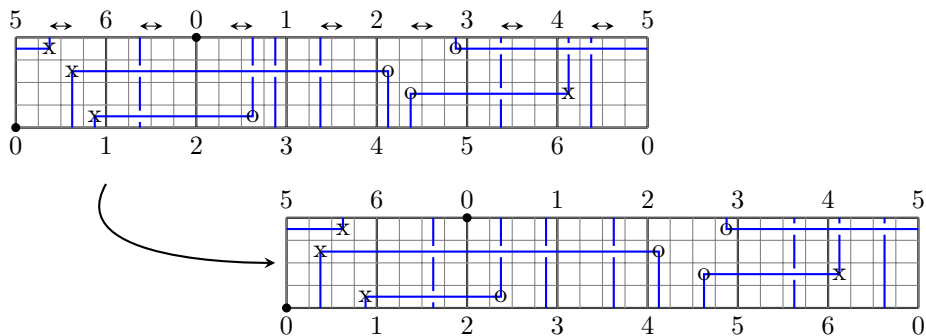

{\bf Commutations:} A commutation interchanges two adjacent columns (or rows) of the grid diagram. Let $A$ be the annulus consisting of the two adjacent columns $c_1,c_2$ (resp.\ rows $r_1,r_2$) involved in the commutation. This annulus is sectioned into $pn$ segments of the $n$ rows (resp.\ columns) of the grid diagram. Let $s_1, s_1'$ be the two segments in $A$ containing the markings of $c_1$ (resp.\ $r_1$). If the markings of $c_2$ (resp.\ $r_2$) are contained in separate components of $A-s_1-s_1'$, the commutation is called \emph{interleaving}. If they are in the same component of $A-s_1-s_1'$ the commutation is called \emph{non-interleaving}. We note that in the literature a commutation typically refers only to what we call a non-interleaving commutation. We have extended the terminology to include the interleaving case. A non-interleaving commutation of columns (resp.\ rows) is a grid move that preserves Legendrian isotopy type \cite{BG}. An interleaving commutation corresponds to a crossing change (see \cite{C}). An example of non-interleaving commutation is shown in Figure \ref{fig:commut}.

Note that a commutation (interleaving or non-interleaving) does not include a column exchange of the type illustrated in Figure \ref{fig:illegalcommut}, where there is a row containing markings of both $c_1$ and $c_2$.

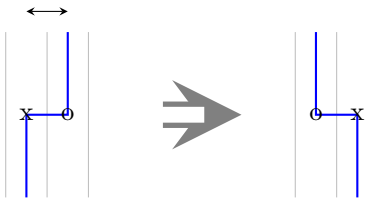
\begin{figure}[ht]
	\begin{tikzpicture}[scale=0.055,>=stealth]
		\foreach \t in {0,70}
		\draw[gray!50!white,thin]
			(0+\t,0) -- (0+\t,40)
			(10+\t,0) --(10+\t,40)
			(20+\t,0) --(20+\t,40);
		\draw[blue,thick]
			(5,0) -- (5,20) -- (15,20) -- (15,40);
		\draw[blue,thick]
			(85,0) --(85,20)-- (75,20) -- (75,40);
		\draw
			(5,20) node {x}
			(15,20)node {o};
		\draw
			(85,20)node {x}
			(75,20)node {o};
	\draw[<->] (5,45) -- (15,45);
	\draw[gray,->,line width=10pt] (63-25,20) -- (82-25,20);
	\draw[line width=6pt,white] (62-25,20) -- (73-25,20);

	\end{tikzpicture}
	\caption{A move which is neither an interleaving nor non-interleaving commutation}
	\label{fig:illegalcommut}
\end{figure}

\subsection{The HOMFLY polynomial in lens spaces}
\label{subsec:HOMpoly}

Let $M$ be a rational homology 3-sphere that is either atoroidal or Seifert-fibered with orientable orbit space. Fix a collection of ``trivial links'' in $M$ having one representative from each homotopy class of links. Kalfagianni and Lin show \cite{KL} that, given a choice of value $J_M(\tau)$ for each trivial link $\tau$ that has no nullhomotopic components, and a choice $J_M(U)$ for the standard unknot $U$, there exists a unique power-series valued invariant $J_M$ with coefficients that are Vassiliev invariants, satisfying the HOMFLY skein relation. 

In \cite{C} the author explicitly defined a collection of \emph{trivial links} in $L(p,q)$ via toroidal grid diagrams. To define these diagrams, let $\mu(P)$ be the signed intersection of an $\alpha$ curve, say $\alpha_0$, with a grid projection $P$ ($\mu(P)$ also represents the homology class of the link associated to $P$). If $P, P'$ are different grid projections for the same grid diagram, then $\mu(P)\equiv\mu(P')\text{ (mod }p)$. Every marking of a grid diagram is in some component of that diagram (recall Definition \ref{defn:compDiag}). Given an $\bb O$ marking $O$, if $C$ is a grid projection of the component containing $O$, we say that $\mu(O)=i$ if $i\equiv\mu(C)\text{ (mod }p)$. 

A \emph{trivial link diagram} $D(\cl I)$ is determined by an index set, which is a $p$-tuple $\cl I=(m_0,m_1,\ldots,m_{p-1})$, where $m_i$ is a non-negative integer for each $i$, in the following manner. Let $n=\sum_{i=0}^{p-1}m_i$. Then $D(\cl I)$ is the unique grid diagram, with grid number $n$, with the following three properties:
\begin{itemize}
	\item all its $\bb O$ markings have coordinates $(\theta_1,\theta_2)=\left(\frac{2i-1}{2n},1-\frac{2i-1}{2n}\right)$ for some $1\le i\le n$;
	\item every component of $D(\cl I)$ has grid number one;
	\item for each $0\le i\le p-1$ there are $m_i$ markings in $\bb O$ with $\mu(O)=i$, and these markings are ordered along the diagonal so that if $\mu(O)\cdot q  \text{ (mod }p)<\mu(O')\cdot q \text{ (mod }p)$, then $\theta_1(O)<\theta_1(O')$.
\end{itemize}
	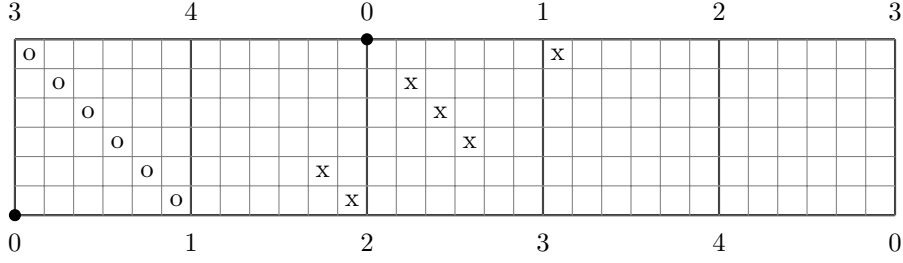
\begin{figure}
		\[\begin{tikzpicture}[scale=0.13,>=stealth]
	\draw[step=18,thick]
			(0,0) grid (90,18);
	\draw[gray,thin,step=3]
			(0,0) grid (90,18);
	\filldraw 
		(0,0) circle (16pt)
		(36,18) circle (16pt);
	\foreach \num in {0,1,2,3}	
	\draw
		(0+\num*18,-1) node[below] {\num}
		(36+\num*18,19) node[above] {\num};
	\draw
		(0,19) node[above] {3}
		(18,19) node[above] {4}
		(72,-1)node[below] {4}
		(90,-1)node[below] {0};
		
	\foreach \x in {(55.5,16.5), (40.5,13.5), (43.5,10.5), (46.5,7.5),(31.5,4.5),(34.5,1.5)}
		\draw 	\x node {x};
	\foreach \o in {(1.5,16.5),(4.5,13.5),(7.5,10.5),(10.5,7.5),(13.5,4.5),(16.5,1.5)}
		\draw	\o node {o};	
		\end{tikzpicture}\]
	\caption{The trivial link diagram $D(\cl I)$ in $L(5,2)$ with $\cl I=(0,1,2,0,3)$.}
	\label{fig:TrivLink}
	\end{figure}

A \emph{trivial link} is any link isotopic to the link associated to a trivial link diagram. Note that any knot admitting a grid number one diagram is a trivial knot. An example of a trivial link diagram is shown in Figure \ref{fig:TrivLink}. The following was shown in \cite{C}:
\begin{thm}
Let $\scr L$ be the set of isotopy classes of oriented links in $L(p,q)$ and let $\scr{TL}\subset\scr L$ denote the set of isotopy classes of trivial links. Define $\scr{TL}^*\subset\scr{TL}$ to be those trivial links with no nullhomologous components. Let $U$ be the isotopy class of the standard unknot, a local knot in $L(p,q)$ that bounds an embedded disk. Suppose we are given a value $J_{p,q}(\tau)\in\ints[a^{\pm1},z^{\pm1}]$ for every $\tau\in\scr{TL}^*$. Then there is a unique map $J_{p,q}:\scr{L}\to\ints[a^{\pm1},z^{\pm1}]$ such that
	\en{
	\item[(i)] $J_{p,q}$ satisfies the skein relation
		\begin{equation*}
		a^{-p}J_{p,q}(L_+)-a^pJ_{p,q}(L_-)=zJ_{p,q}(L_0).
		\end{equation*}
	\item[(ii)] $J_{p,q}(U)=a^{-p+1}$.
	\item[(iii)] $J_{p,q}\left(U\coprod L\right)=\frac{a^{-p}-a^p}{z}J_{p,q}(L)$.
	}

\label{thm:HOMpoly}
\end{thm}

Theorem \ref{thm:HOMpoly} provides a HOMFLY polynomial for links in a lens space. In order to prove Corollary \ref{BennBd} we choose a normalization of $J_{p,q}$ so that the inequality of Corollary \ref{BennBd} is at least satisfied on $\scr{TL}$. If $\tau\in\scr{TL}^*$ has trivial link diagram $D(\cl I)$, then we normalize so that $J_{p,q}(\tau)=a^{p\cdot\slk_\rat(T(\tau)))+1}$ where $T(\tau)$ is the transverse pushoff of the Legendrian link associated to $D(\cl I)$ defined at the end of Section \ref{subsec:ratLegLinks}. 

In the course of proving this theorem the author described a skein theory that deals only with grid diagrams. A crossing in this skein theory, called a \emph{skein crossing}, is a pair of adjacent columns of the grid diagram that are interleaving. A \emph{skein crossing change} is made by commutation of the two adjacent columns in the skein crossing and a \emph{resolution} of a skein crossing is achieved by interchanging two of the markings in the pair of columns that are of the same type: either both are in $\bb X$ or both are in $\bb O$ (if one interchanges the $\bb O$ markings instead of those in $\bb X$, the difference is a non-interleaving commutation, which is an isotopy). That is, our skein relations involve diagrams $D_+$, $D_-$ and $D_0$ that differ only at a pair of adjacent columns, where they differ as in Figure \ref{fig:skeinTheory}. Through the course of the paper a triple of links $L_+, L_-$, and $L_0$ is a set of three links that (up to isotopy) admit grid diagrams $D_+, D_-$, and $D_0$ respectively.
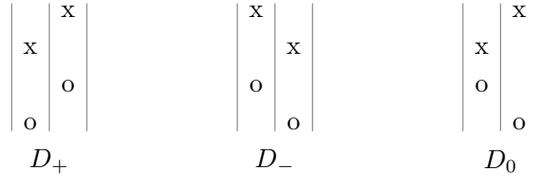
\begin{figure}[h]
	\[\begin{tikzpicture}
\foreach \t in {0,3,6}
	\draw[gray,thin]
		(\t-0.25,-0.1) -- (\t-0.25,1.6)
		(\t+0.25,-0.1) -- (\t+0.25,1.6)
		(\t+0.75,-0.1)-- (\t+0.75,1.6);
	\draw
		(0,0) node {o}
		(0,1) node {x}
		(0.5,0.5) node {o}
		(0.5,1.5) node {x}
		(0.25,-0.5) node {$D_+$};
	\draw
		(3,0.5) node {o}
		(3,1.5) node {x}
		(3.5,0) node {o}
		(3.5,1) node {x}
		(3.25,-0.5) node {$D_-$};
	\draw
		(6,0.5) node {o}
		(6,1) node {x}
		(6.5,0) node {o}
		(6.5,1.5) node {x}
		(6.25,-0.5) node {$D_0$};
	\end{tikzpicture}\]
\caption{A positive, negative, and resolved skein crossing on a grid diagram.}
\label{fig:skeinTheory}
\end{figure}

The proof of Theorem \ref{thm:HOMpoly} also relies on an important lemma which reveals how to reduce the complexity of a grid diagram in an understood way. Since we will also use this lemma to prove the FWM inequality in $L(p,q)$, we restate it here.

\begin{lem}[Lemma 4.4, \cite{C}] Let $K$ be a Legendrian link in $L(p,q)$ associated to a grid diagram $D_K$. Suppose $D_K$ has a component with grid number more than $1$. Then there exists a sequence of commutations followed by a destabilization giving a new grid diagram $D'$ such that $GN(D')<GN(D_K)$, where $GN(D)$ denotes the grid number of a grid diagram $D$. 
\label{lemRuth}
\end{lem}

We remark that the sequence of commutations in Lemma \ref{lemRuth} may have both interleaving and non-interleaving commutations, and that a commutation cannot involve two columns that have markings in the same row. Further, it was seen in the proof that one can ensure that none of the row commutations in the sequence are interleaving, and thus the only commutations that are not Legendrian isotopy are skein crossing changes. It was also shown in \cite{C} that the links associated to grid diagrams $D_+, D_-,$ and $D_0$ form a skein triple.

Finally, a complexity $\psi$ was defined in \cite{C} on grid diagrams in $L(p,q)$. We will not review how to define $\psi$ here, but wish to make two remarks. First, $\psi$ is minimized (within a homotopy class) on certain diagrams associated to trivial links. Second, given any grid diagram $D_{\pm}$ with a skein crossing, the complexity $\psi$ decreases under the resolution $D_{\pm}\rightsquigarrow D_0$.
\subsection{Rationally null-homologous Legendrian links}
\label{subsec:ratLegLinks}

We review here the definitions given in \cite{BE} (see also \cite{BG}) of the Thurston-Bennequin number and rotation number, and for transverse links the self-linking number, in the setting of a rationally null-homologous link. 

To begin requires the notion of a \emph{rational Seifert surface}. Let $K$ be an oriented rationally null-homologous knot in a manifold $M$ and let $r$ be the order of the homology class of $K$ in $H_1(M,\ints)$. Then, writing $\nu(K)$ for a normal neighborhood of $K$, for some $s\in\ints$ there is a curve $\gamma$ of slope $(r,s)$ on $\bd\nu(K)$ that bounds an oriented surface $\Sigma^0\subset M\setminus\nu(K)$. Let $\Sigma$ be the union of $\Sigma^0$ and the cone in $\nu(K)$ of $\gamma$ to $K$. Figure \ref{qSfcConst} depicts a meridional cross-section of $\nu(K)$ and the cone  of $\gamma$ to $K$.

\begin{figure}[ht]
	\[\begin{tikzpicture}[>=stealth]
		\draw[very thick] (0,0) ellipse (3 and 2);
		\draw[thick] 
			(-1,0.2) 	.. controls (-1,-.3) and (1,-.3) .. (1,0.2)
			(-0.8,0) 	.. controls (-0.8,0.3) and (0.8,0.3) .. (0.8,0);
			
		\draw[red,thick,scale=1.2]	(0,-1.5) 
			.. controls (-2.2,-1.25) and (-2.4,-0.3) .. (-2,0.4)
			.. controls (-1.5,1.1) and (-0.5,1.22) .. (0,1.2)
			.. controls (1.4,1.17) and (1.3,0.25) .. (0.75,0);
		\draw[dashed,red,scale=1.2]	(0.55,-0.075)
			.. controls (-2.75,-0.75) and (-1,2.5) .. (1.4,1.38);
		\draw[red,thick,scale=1.2]		(1.4,1.38)
			.. controls (2.5,1) and (2.7,-1.15) .. (0,-1)
			.. controls (-2.6,-0.85) and (-1.5,1.1) .. (-0.4,0.5)
			.. controls (0.04,0.14) and (-0.17,0.25) .. (0,0.185);
		\draw[red,scale=1.2,dashed]		(0,0.185)
			.. controls (0.27,0.1) and (0.56,0.35) .. (1,0.4)
			.. controls (2,0.55) and (2.8,-0.3) .. (1.8,-1.18);
		\draw[red,thick,scale=1.2]	(1.8,-1.18)
			.. controls (0.85,-1.6) and (0.43,-1.53) .. (0,-1.5);	
		\draw[dashed,thick]	(1.9,0) ellipse (1.1 and 0.5);
		\filldraw	(1.9,0) circle (1pt);
		\draw	(0,0) ellipse (1.9 and 1);
		
		\draw
		(1.9,0.5) -- (6.23,1.47)
		(1.9,-0.5) -- (6.23,-1.47);
		\draw[dashed,thick]	(6.5,0) circle (1.5);
		\draw[red,thick]
			(5.6,1.2) -- (6.5,0) -- (7.4,1.2)
			(6.5,0) -- (6.9,-1.45);
		\filldraw
			(6.5,0) circle (2pt);
		\filldraw[red]
			(5.6,1.2) circle (2pt)
			(7.4,1.2) circle (2pt)
			(6.9,-1.45) circle (2pt);
		\draw 	(0,-1) node[above] {$K$}
				(6.5,0) node [right] {$K$}
				(-1,-1.5) node {$\gamma$};
		\draw[<-]
			(7,0.7) -- (6,1.75);
		\draw
			(6,1.75) node[above] {cone of $\gamma$ to $K$};
		
	\end{tikzpicture}\]
\caption{Constructing a rational Seifert surface}
\label{qSfcConst}
\end{figure}
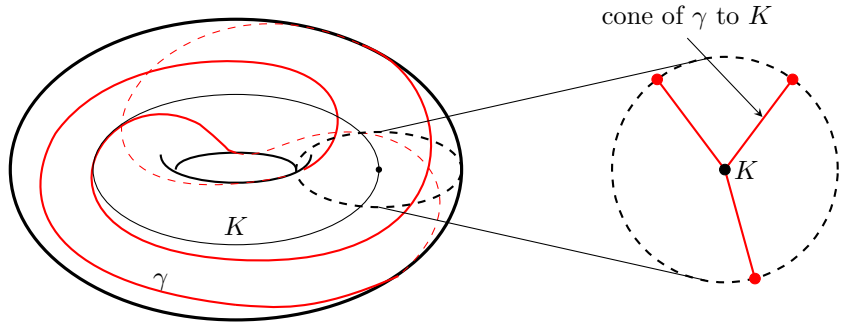

Note that the interior of $\Sigma$ is an embedded surface in $M$ and $\bd\Sigma$ is an $r$-fold cover of $K$.

\begin{defn} Let $K$ be an oriented, rationally null-homologous knot in $M$ with order $r$ in $H_1(M,\ints)$. A \emph{rational Seifert surface for $K$} is an oriented surface $\Sigma$ with a map $j:\Sigma\to M$ such that $j$ is an embedding on the interior of $\Sigma$, $j(\bd\Sigma)=K$, and $j\vert_{\bd\Sigma}$ is an $r$-fold cover of $K$.
\end{defn}

The previous discussion shows that every oriented rationally null-homologous knot has a rational Seifert surface. Often we will abuse terminology and understand a Seifert surface for a rationally null-homologous knot to be a rational Seifert surface. We can now define the classical invariants for Legendrian and transverse knots in the case of rationally null-homologous knots.

\begin{defn}
Let $K$ be an oriented, rationally null-homologous knot with order $r$ as above and let $j:\Sigma\to M$ be a Seifert surface for $K$.
\en{
	\item Given another oriented knot $K'$, define 
		\[lk_\rat(K,K') = \frac1r\Sigma\cdot K'.\]
	\item If in addition $K$ is a Legendrian knot in $(M,\xi)$ and $K'$ is a pushoff of $K$ in the direction of the contact framing given by $\xi\vert_{K}\cap\nu(K)$, then define the \emph{(rational) Thurston-Bennequin number} of $K$ by
				\[\tb_\rat(K) = lk_\rat(K,K').\]
	\item For $x\in K$ let $dK_x$ denote the tangent vector to $K$ at $x$. For $K$ Legendrian as above, $dK$ is a section of the bundle $\xi_K$. Take a trivialization $j^*\xi_\Sigma\cong\rls^2\times\Sigma$. Define the \emph{(rational) rotation number} of $K$ to be the winding number of $j^*dK$ in $\rls^2$ under this trivialization, divided by $r$:
		\[\rot_\rat(K)=\frac1r\text{winding}(j^*dK,\rls^2).\]
	\item If $K$ is a transverse knot in $(M,\xi)$, let $v$ be a non-zero section of $j^*\xi$. Normalize $v$ so that $v\vert_{\bd\Sigma}$ defines a curve $K'$ in $\bd\nu(K)$. Define the \emph{(rational) self-linking} of $K$ to be
		\[\slk_\rat(K)=lk_\rat(K,K').\]
}
\label{defn:classicalInvts}
\end{defn}

\begin{rem}
In general, the rotation number of a Legendrian knot and the self-linking number of a transverse knot depend on the relative homology class of $\Sigma$. Yet once this class is fixed, they do not depend on other choices made {--} the trivialization of $j^*\xi_{\Sigma}$ in the case of $\rot_\rat$ and the section $v$ in the case of $\slk_\rat$ (see \cite{BE}). The dependence on the homology of $\Sigma$ does not apply to lens spaces, however, since $H_2(L(p,q))=0$.
\end{rem}

In the case of null-homologous links (e.g. Legendrian links in $(S^3,\xi_{st})$) these numbers are known as the ``classical invariants'' of Legendrian and transverse links. In this case the classical invariants are always integers. However, in the case of rationally null-homologous links, these numbers are generally rational.

We recall that for any Legendrian knot $K$ in a contact manifold $(M,\xi)$, there is a related transverse knot $T_+(K)$ called the \emph{positive transverse push-off of} $K$ (see \cite{B}). Using this construction on each component of a link, we can get the positive transverse push-off of a Legendrian link. To construct $T_+(K)$, find a tubular neighborhood of $K$ that is contactomorphic, for sufficiently small $\ve$, to $C_\ve=\setn{[(x,y,z)]}{y^2+z^2<\ve^2, x=x+1}$: the quotient of an $\ve$-neighborhood of the $x$-axis in $(\rls^3,\ker(dz-ydx))$ by the action $x\mapsto x+1$. Here the orientation of the image of $K$ under the contactomorphism is in the direction of increasing $x$-values. The positive transverse push-off $T_+(K)$ is defined to be the image of $\set{(x,\ve/2,0)}$ in the neighborhood of $K$. We can also define the negative transverse push-off $T_-(K)$ to be the image of $\set{(x,-\ve/2,0)}$. Baker and Etnyre show the following \cite{BE}:

\begin{lem}
Let $K$ be a rationally null-homologous Legendrian link and let $T_\pm(K)$ be defined as above. Then
	\[\emph{slk}_\rat(T_{\pm}(K)) = \emph{tb}_\rat(K)\mp\emph{rot}_\rat(K).\]
\label{lem:slkTOtb-rot}
\end{lem}
\begin{rem} We will denote by $T(K)$ whichever of the transverse push-offs of $K$ has $\slk_\rat(T(K))=\tb_\rat(K)+\abs{\rot_\rat(K)}$. If $\rot_\rat(K)=0$ then $T(K)$ is taken to be the positive transverse push-off.
\label{rem:T(K)}
\end{rem}

\section{Formulas for $\tb_\rat(K), \rot_\rat(K),$ and $\slk_\rat(K)$ from a grid projection in $L(p,q)$}
\label{sec:gridProjForms}

A method for computing the (rational) Thurston-Bennequin number of a Legendrian link in $L(p,q)$ via the Maslov index of a corresponding grid diagram was given in \cite{BG}. The complexity of such computations increases quickly as the grid number of the diagram increases. 

We recall that in $(S^3,\xi_{st})$ there are formulas for the classical invariants that can be computed from a front projection (see \cite{B}). The formulas we give here for computing the (rational) Thurston-Bennequin, rotation, and self-linking numbers in $L(p,q)$ are in the same spirit. In fact, they are derived from the former. 

\begin{defn}
Given a grid projection $P$ for an oriented link in $L(p,q)$ denote the writhe of the projection by $w(P)$ and the number of cusps of the projection by $c(P)$. Also, let $\mu(P)$ denote the algebraic intersection number of $\alpha_0$ with $P$ and $\lambda(P)$ the algebraic intersection number of $P$ with $\beta_0$. 
\label{defn:diagramInvts}
\end{defn}

Let $K$ be the link associated to a given grid projection $P$ on a grid diagram. We recall that for a given row (resp.\ column) in that diagram, $P$ contains one of the two choices of horizontal (resp.\ vertical) arcs. A projection $P'$ which is identical to $P$ except that it contains the other arc in this row (resp.\ column) corresponds to a link $K'$ that differs from $K$ by a Legendrian isotopy across a meridian of $V^\alpha$ (resp.\ $V^\beta$). We call this isotopy a \emph{disk slide}.

Recall that in $(S^3,\xi_{st})$ if $P$ is the front projection of a Legendrian link $K$ then $\tb(K)=w(P)-\frac12c(P)$. Moreover, we note that $(S^3,\xi_{st}) = (L(1,0),\xi_{UT})$, and if a grid projection $P$ in $L(1,0)$ is contained in a planar subset of $T$ then there is a slight perturbation of $P$ giving a front projection (see \cite{BG}).

\begin{prop} Let $K$ be an oriented Legendrian link in $(S^3,\xi_{st})=(L(1,0),\xi_{UT})$ and let $P$ be a grid projection of $K$. Let $l=\lambda(P)$ and $m=\mu(P)$. Then
			\[\emph{tb}(K)=w(P)-\frac12c(P)-ml.\]
\label{tbS3form}
\end{prop}

\begin{proof}
If $P$ is a projection on a planar subset of $\Sigma$ then we can consider it as a front projection. In this case $l=m=0$ and the proposition follows immediately. 
	
Suppose $P'$ is any grid projection of $K$ for which the proposition holds. We prove that if $P$ differs from $P'$ by a disk slide (either a disk with boundary parallel to $\alpha_0$ or along one with boundary parallel to $\beta_0$) then the proposition holds for $P$ as well. Since every grid projection of $K$ is related to a planar grid projection by a sequence of disk slides, this will prove the proposition. Inherent in this proof is the fact that disk slides correspond to Legendrian isotopy \cite{BG}.
	
We denote the algebraic intersection of two oriented curves $\gamma,\delta$ on $\Sigma$ that meet transversely by $\inp{\gamma}{\delta}$. Note that for any circle $c$ on $\Sigma$, parallel to $\alpha_0$ and transverse to $P'$, we have $m':=\inp{\alpha_0}{P'}=\inp c{P'}$, since $c$ separates $\Sigma\setminus\alpha_0$.
		
\begin{figure}[ht]	
	\[\begin{tikzpicture}[>=stealth,scale=1.3]
		\draw[thick]
			(2,2) -- (2,1) -- node[below] {$a'$} (6,1) -- (6,0);
		\draw[thick,cm={1,0,0,1,(0,-2.5)}]
			(2,0)--(2,1) -- node[below]{$a'$} (6,1)--(6,2);
		\draw[thick,cm={1,0,0,1,(0,-4.5)}]
			(2,0) -- (2,1) -- node[below] {$a'$} (6,1) -- (6,0);
	\foreach \ty in {0,-2.5,-4.5}
		\draw[thick,->,cm={1,0,0,1,(0,\ty)}]
			(0,1.25) -- node[above] {$c$} (8,1.25);
	\foreach \ty in {0,-2.5,-4.5}
		\draw[dashed,thick,cm={1,0,0,1,(0,\ty)}]
			(0,1) -- node[below] {$a$} (2,1);
	\foreach \ty in {0,-2.5,-4.5}
		\draw[dashed,thick,cm={1,0,0,1,(0,\ty)}]
			(6,1) -- node[below] {$a$} (8,1);
		\draw
			(-0.5,1.7) node {{\bf \emph{Case 1}}}
			(-0.5,-0.8) node {{\bf \emph{Case 2}}}
			(-0.5,-2.8) node {{\bf \emph{Case 3}}}
			(1,0) node {$cusp$}
			(7,-0.15) node {$cusp$}
			(7,-4.65) node {$cusp$};
		\draw[->]
			(1.4,0) .. controls (2.3,0.2) and (2.2,0.3) .. (2,0.85);
	\foreach \ty in {0,-4.5}
		\draw[->,cm={1,0,0,1,(0,\ty)}]
			(7,0) .. controls (6.8, 0.4) and (6.7,0.6) .. (6.15,0.85);
	\end{tikzpicture}\]
\caption{How $\tb$ changes with a disk slide}
\label{fig:tbFormS3}
\end{figure}
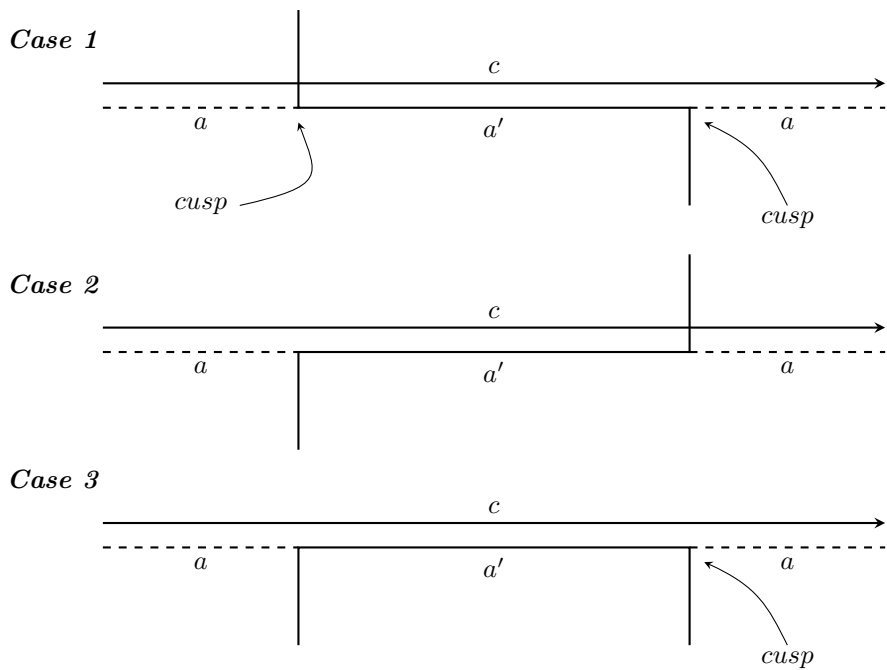

Suppose $P$ differs from $P'$ by a disk slide. Suppose further that the disk slide is along a disk $\Delta$ with boundary parallel to $\alpha_0$. There is some orientation of $\Delta$ such that $\bd\Delta=(-a')\cup a$, where $a'$ is an arc of $P'$ and $a$ is the arc that replaces it in the projection $P$ (see Figure \ref{fig:tbFormS3}). Let $c$ be a circle that is parallel and coherently oriented to $\alpha_0$ and a small distance away from $(-a')\cup a$. Then $c$ intersects $P'$ and $P$ tranversely. The writhe of $P$ differs from that of $P'$ only by the difference of writhe along double points of $a'$ and double points of $a$. Thus if $a'$ is an arc with two cusps as in case 1 of Figure \ref{fig:tbFormS3}, then
		\[\pm(w(P)-w(P')+1)=\inp c{P'}=m',\]
where the sign on the left depends on the orientation of $a$ with respect to $c$. If $a'$ has no cusps as in case 2 of Figure \ref{fig:tbFormS3} then
		\[\pm(w(P)-w(P')-1)=\inp c{P'}=m'.\]
Finally, if $a'$ is an arc with one cusp as in case 3 then
		\[\pm(w(P)-w(P'))=\inp c{P'}=m',\]
with the sign depending again on the orientation of $a$. Also, if $a'$ is as in case 1 then $c(P)=c(P')-2$, if we are in case 2 then $c(P)=c(P')+2$, and if we are in case 3, then $c(P)=c(P')$. 

Since $l=l'\pm1$ (where $l'=\lambda(P')$) and $m=m'$ we can now check that
	\[w(P)-\frac12c(P)-ml= w(P')-\frac12c(P')-m'l'=\tb(K)\]
for cases 1{--}3.

The argument is analogous if the disk slide is along a meridional disk.
\end{proof}
\begin{cor}
Let $K\subset (L(p,q),\xi_{UT})$ be an oriented Legendrian link with $P$ a grid projection for $K$. Let $\lambda=\lambda(P)$ and $\mu=\mu(P)$. Then
	\[\emph{tb}_\rat(K)=w(P)-\frac12c(P)-\frac{\mu\lambda}p.\]
\label{tbForm}
\end{cor}
\begin{proof} 
Lift the projection $P$ to a grid projection $\wt P$ in $S^3$ by lifting the Heegaard torus of $L(p,q)$ to $\Sigma$. Then clearly $pw(P)=w(\wt P)$ and $pc(P)=c(\wt P)$. Moreover, $\mu(K)=\mu(\wt K)$ and $\lambda(K)=\lambda(\wt K)$ by definition.

It was shown in \cite{BG} that $\tb_\rat(K)=\frac{\tb(\wt K)}p$, and this completes the proof.
\end{proof}

\begin{cor}
Let $K\subset (L(p,q),\xi_{UT})$ be a Legendrian knot with its contact framing. If there is an integral surgery on $K$ that gives a homology sphere  $S$ then $\mu(K)^2\equiv\pm q'\mod p$, where $qq'\equiv1\mod p$. In particular, if there is a knot in $L(p,q)$ on which integer surgery yields $S^3$ then $\pm q$ is a quadratic residue mod $p$.
\label{lem:quadResThm}
\end{cor}
\begin{rem}
We note that the last statement of the corollary is a well-known result of Fintushel and Stern \cite{FS}.
\label{rem:quadResThm}
\end{rem}
\begin{proof}
	In \cite{BG} it is shown that $p\cdot\tb_\rat(K)\equiv\pm1\mod p$. With Corollary \ref{tbForm} this implies that $\mu\lambda\equiv\pm1\mod p$. Since the projection of $K$ is a closed curve, $\lambda\equiv \mu q\mod p$, so that $\mu^2q\equiv\pm1\mod p$.
\end{proof}

The following proposition shows how to compute $\rot_\rat(K)$ given a projection $P$ corresponding to a grid diagram for an oriented Legendrian link $K$ in $(L(p,q),\xi_{UT})$. Its proof is similar to the proof of Proposition \ref{tbS3form} and Corollary \ref{tbForm}.

\begin{prop} Let $\mu$ and $\lambda$ be as in Corollary \ref{tbForm}. Note that each cusp in a grid projection is comprised of a horizontal and a vertical arc. Define $c_u(P)$ to be the number of cusps whose horizontal arc is oriented against $\alpha_0$ and $c_d(P)$ to be the number of cusps with horizontal arc oriented in the direction of $\alpha_0$. Then
	\[\emph{rot}_\rat(K)=\frac12\left(c_d(P)-c_u(P)\right)-\frac{(\lambda-\mu)}p.\]
\label{rotForm}
\end{prop}

\begin{cor} Let $T_+(K)$ (resp.\ $T_-(K)$) be the positive (resp.\ negative) transverse pushoff of the Legendrian link $K$ in $(L(p,q),\xi_{UT})$. Then
	\al{
	\emph{sl}_\rat(T_+(K))&=w(P)-c_d(P)-\frac{\mu\lambda+(\mu-\lambda)}p\quad\text{and}\\
	\emph{sl}_\rat(T_-(K))&=w(P)-c_u(P)-\frac{\mu\lambda+(\lambda-\mu)}p,
	}
where $P$ is the projection of $K$ as above.
\label{slForm}
\end{cor}
\begin{proof} This results from Corollary \ref{tbForm}, Proposition \ref{rotForm}, and Lemma \ref{lem:slkTOtb-rot}.
\end{proof}

\section{Bennequin-type bounds in $L(p,q)$}
\label{sec:BennBd}

In this section we prove the lens space analogue of the version of the FWM inequality observed by Fuchs and Tabachnikov \cite{FT}. To achieve this, we prove the analogue of a theorem of Lenny Ng (Theorem 1 in \cite{Ng}) for oriented links in the contact lens space $(L(p,q),\xi_{UT})$. While the proof follows the ideas of Ng in \cite{Ng}, we note that it is independent of the work of Rutherford \cite{Ruth} that plays a key role in the proof given in \cite{Ng}. It is an example of the power of the point of view of grid diagrams in understanding Legendrian links.

{\bf Theorem \ref{mainThm2}}.\quad
{\it Let $i$ be a $\rat$-valued invariant of oriented links in $L(p,q)$ such that
	\al{
		&i(L_+)+1\le\max(i(L_-)-1,i(L_0))\\
	and\\
		&i(L_-)-1\le\max(i(L_+)+1,i(L_0))
	}
where $L_+, L_-,$ and $L_0$ are oriented links that differ as in the skein relation. If $\emph{sl}_\rat(T(\tau))\le-i(\tau)$ for every trivial link $\tau$ in $L(p,q)$, then
		\[\overline{\emph{sl}}_\rat(L)\le-i(L)\]
	 for every link $L$ in $L(p,q)$. Here $\overline{\emph{sl}}_\rat$ is the maximum self-linking number among transverse links in $(L(p,q),\xi_{UT})$ that are isotopic to $L$.
}
\medskip 

\begin{proof}
	In the proof we will abuse notation, often using $P$ to refer to both a grid projection and the Legendrian link it specifies. Our strategy of proof is as follows. Assume our link is a positive transverse push-off $T_+(P)$. Define $\wt i(P)=i(P)+w(P)$, where $w(P)$ is the writhe of $P$. Then by the formula of Corollary \ref{slForm}, the inequality $\slk_\rat(T_+(P))\le-i(P)$ is equivalent to 
	\begin{equation}
	-\slk_\rat(T_+(P))-i(P)=c_d(P)+\frac{\mu\lambda+(\mu-\lambda)}p-\wt i(P)\ge0.
	\label{slIneq}
	\end{equation}
	
	By the results from \cite{C}, in particular Lemma \ref{lemRuth}, there is a minimal length sequence 
		\[P\overset{\alpha_1}{\to}P_1\overset{\alpha_2}{\to},\ldots,\overset{\alpha_n}{\to}P_n=P_\tau\]
	taking $P$ to a grid projection $P_\tau$ associated to a trivial link diagram, where each $\alpha_i$ is either a skein crossing change, Legendrian isotopy, or destabilization. We will show that if some $\alpha_i$ increases the left side of inequality (\ref{slIneq}) then it is a skein crossing change, and that in this case the resolution $P\rightsquigarrow P_0$ does not increase the left side of (\ref{slIneq}). If $P, P_0$ have underlying diagrams $D, D_0$, then $\psi(D)>\psi(D_0)$ (recall $\psi$ from the end of Section \ref{subsec:HOMpoly}). As $\psi$ is minimized on diagrams for trivial links, we may assume inductively that inequality (\ref{slIneq}) holds for $P_0$, since $\slk_\rat(T_+(\tau))\le\slk_\rat(T(\tau))\le-i(\tau)$ for any trivial link $\tau$. 
	
	If instead our transverse link is $T_-(P)$ for some $P$ then the same argument can be carried through, replacing inequality (\ref{slIneq}) with the corresponding formula. Since every transverse link is some transverse push-off $T_\pm(P)$ of a Legendrian link, this argument is sufficient to prove the theorem.
	
	 The expression on the left side of (\ref{slIneq}) is a Legendrian invariant, so if $\alpha_i$ is a Legendrian isotopy the expression is unchanged. In the case of a destabilization we make the following claim.
	
	{\bf Claim 1:} If $\alpha_i$ is a destabilization, then $-\slk_\rat(T_+(P_{i-1}))-i(P_{i-1})\ge -\slk_\rat(T_+(P_i))-i(P_i)$, so $\alpha_i$ does not increase the left side of (\ref{slIneq}). 

	\begin{proof}[Proof of Claim 1]
	
	First, $\mu$ and $\lambda$ are each an algebraic intersection of $P$ with some $\alpha$-curve or $\beta$-curve respectively. Therefore they are unchanged by a destabilization.

\begin{figure}[ht]
	\[\begin{tikzpicture}[>=stealth,scale=1.5]
		\draw[thick]
			(0.5,0.5) node {o}
			(0.5,1) node {x}
			(1,0.5) node {x}
			(4,0.5) node {x};
		\draw[blue,very thick,->]
			(1,1.8) -- (1,1.1)
			(1,0.9) -- (1,0.5) -- (0.5,0.5) -- (0.5,1) -- (1.8,1);
		\draw[blue,very thick,->]
			(4,1.8) -- (4,0.5) -- (5.8,0.5);
	
		\draw[gray,->,line width=10pt] (2.3,1) -- (3.2,1);
		\draw[line width=6pt,white] (2.2,1) -- (2.7,1);
		\draw
			(-0.2,0.2) node[below]{{\footnotesize upward cusp}}
			(-0.2,1.3) node[above] {{\footnotesize negative crossing}}
			(4.7,0.2) node[below] {{\footnotesize downward cusp}};
		\draw[->]
			(-0.2,0.2) -- (0.5,0.4);
		\draw[->]
			(-0.2,1.3) -- (0.95,1.05);
		\draw[->]
			(4.7,0.2) -- (4,0.4);
	\end{tikzpicture}\]
\caption{On an $X:NE$ destabilization $c_d(P)$ increases.}
\label{c_dInc}
\end{figure}
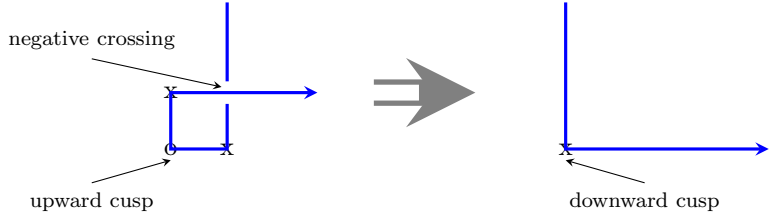

	It is possible that for some destabilization (that is not Legendrian isotopy), $c_d(P)$ changes. If $c_d(P)$ decreases, then $-\slk_\rat(T_+(P))$ decreases. The destabilization was an isotopy, so $i(P)$ remains unchanged and we see that the left side of (\ref{slIneq}) does not increase. If $c_d(P)$ increases the destabilization is of type X:NE or O:SW. Figure \ref{c_dInc} depicts when the destabilization is type X:NE. In such a case, $c_d(P)$ increases by 1, but $w(P)$ does also. Thus $\slk_\rat(T_+(P))$ does not change, and so $-\slk_\rat(T_+(P))-i(P)$ also does not change.
	\end{proof}
	
	We remark that Claim 1 also holds for the case of a negative push-off $T_-(P)$. In the proof one only need consider $c_u(P)$ instead of $c_d(P)$.
	
	Thus we only need concern ourselves with skein crossing changes. Suppose $\alpha_1$ is a skein crossing change. If the pair of columns involved in $\alpha_1$ are a positive (resp.\ negative) skein crossing in $P$, then use disk slides if necessary so that, at the skein crossing, $P$ appears as in Figure \ref{locSkeinTree} (resp.\ as in Figure \ref{locSkeinTree2}). Since disk slides are Legendrian isotopy, this does not alter the left side of inequality (\ref{slIneq}). 
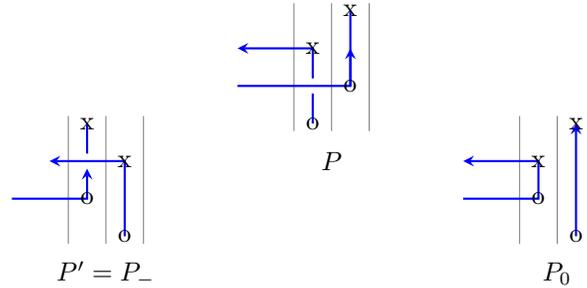
\begin{figure}[h]
	\[\begin{tikzpicture}[>=stealth]
\foreach \t in {0,3}
	\draw[gray, thin]
		(\t-0.25,-0.1+\t/2) -- (\t-0.25,1.6+\t/2)
		(\t+0.25,-0.1+\t/2) -- (\t+0.25,1.6+\t/2)
		(\t+0.75,-0.1+\t/2)-- (\t+0.75,1.6+\t/2);
	\draw[gray,thin]
		(6-0.25,-0.1) -- (6-0.25,1.6)
		(6.25,-0.1) -- (6.25,1.6)
		(6.75,-0.1) -- (6.75,1.6);

	\draw[blue,thick,cm={1,0,0,1,(3,1.5)}]
		(0,0) -- (0,0.4)
		(-1,0.5) -- (0.5,0.5) -- (0.5,1.5);
	\draw[blue,thick,->,cm={1,0,0,1,(3,1.5)}]
		(0,0.6) -- (0,1) -- (-1,1);
	\draw[blue,thick,->,cm={1,0,0,1,(3,1.5)}]
		(0.5,0.5) -- (0.5,1);
	\draw[cm={1,0,0,1,(3,1.5)}]
		(0,0) node {o}
		(0,1) node {x}
		(0.5,0.5) node {o}
		(0.5,1.5) node {x}
		(0.25,-0.5) node {$P$};

	\draw[blue,thick,->,cm={1,0,0,1,(-3,-1.5)}]
		(2,2) -- (3,2) -- (3,2.4);
	\draw[blue,thick,->,cm={1,0,0,1,(-3,-1.5)}]
		(3.5,3/2) -- (3.5,2.5) -- (2.5,2.5);
	\draw[blue,thick,cm={1,0,0,1,(-3,-1.5)}]
		(3,2.6) -- (3,3);
	\draw[cm={1,0,0,1,(-3,-1.5)}]
		(3,2) node {o}
		(3,3) node {x}
		(3.5,1.5) node {o}
		(3.5,2.5) node {x}
		(3.25,1) node {$P'=P_-$};

	\draw[blue,thick,->]
		(5,0.5) -- (6,0.5) -- (6,1) -- (5,1);
	\draw[blue,thick,->]
		(6.5,0) -- (6.5,1.5);
	\draw
		(6,0.5) node {o}
		(6,1) node {x}
		(6.5,0) node {o}
		(6.5,1.5) node {x}
		(6.25,-0.5) node {$P_0$};
	\end{tikzpicture}\]
\caption{A skein crossing change and resolution at a positive skein crossing.}
\label{locSkeinTree}
\end{figure}
	
	If $\alpha_1$ does not increase the left hand side of (\ref{slIneq}), then we are finished by induction. Otherwise the left side of (\ref{slIneq}) does increase in passing from the projection $P$ to $P_-$ (resp.\ $P_+$).
	
	\begin{lem} Let $P$ be a grid projection containing a skein crossing as in Figure \ref{locSkeinTree} or as in Figure \ref{locSkeinTree2}, and let $P'$ be the grid projection obtained by a skein crossing change at that crossing. Further, suppose $-\emph{sl}_\rat(T_+(P))-i(P)<-\emph{sl}_\rat(T_+(P'))-i(P')$. Then $-\emph{sl}_\rat(T_+(P))-i(P)\ge-\emph{sl}_\rat(T_+(P_0))-i(P_0)$. 
	\label{lem:passingThru}
	\end{lem}
	As we remarked before, there is a complexity $\psi$ which is minimized by diagrams for trivial links, such that if $D_0$ is the grid diagram for $P_0$ and $D$ the diagram for $P$, then $\psi(D_0)<\psi(D)$. Thus we may assume that $P_0$ satisfies (\ref{slIneq}), so by Lemma \ref{lem:passingThru}, $P$ does also.
\end{proof}
	
\begin{figure}[ht]
	\[\begin{tikzpicture}[>=stealth]
\foreach \t in {0,3}
	\draw[gray, thin]
		(\t-0.25,-0.1+\t/2) -- (\t-0.25,1.6+\t/2)
		(\t+0.25,-0.1+\t/2) -- (\t+0.25,1.6+\t/2)
		(\t+0.75,-0.1+\t/2)-- (\t+0.75,1.6+\t/2);
	\draw[gray,thin]
		(6-0.25,-0.1) -- (6-0.25,1.6)
		(6.25,-0.1) -- (6.25,1.6)
		(6.75,-0.1) -- (6.75,1.6);

	\draw[blue,thick]
		(0,0) -- (0,0.4)
		(-1,0.5) -- (0.5,0.5) -- (0.5,1.5);
	\draw[blue,thick,->]
		(0,0.6) -- (0,1) -- (-1,1);
	\draw[blue,thick,->]
		(0.5,0.5) -- (0.5,1);
	\draw
		(0,0) node {o}
		(0,1) node {x}
		(0.5,0.5) node {o}
		(0.5,1.5) node {x}
		(0.25,-0.5) node {$P'=P_+$};

	\draw[blue,thick,->]
		(2,2) -- (3,2) -- (3,2.4);
	\draw[blue,thick,->]
		(3.5,3/2) -- (3.5,2.5) -- (2.5,2.5);
	\draw[blue,thick]
		(3,2.6) -- (3,3);
	\draw
		(3,2) node {o}
		(3,3) node {x}
		(3.5,1.5) node {o}
		(3.5,2.5) node {x}
		(3.25,1) node {$P$};

	\draw[blue,thick,->]
		(5,0.5) -- (6.5,0.5) -- (6.5,1) -- (5,1);
	\draw[blue,thick]
		(6,0) -- (6,0.4) (6,0.6)--(6,0.9) (6,1.1) -- (6,1.5);
	\draw[blue,thick,->]
		(6,1.1) -- (6,1.3);
	\draw
		(6.5,0.5) node {o}
		(6.5,1) node {x}
		(6,0) node {o}
		(6,1.5) node {x}
		(6.25,-0.5) node {$P_0$};
	\end{tikzpicture}\]
\caption{A skein crossing change and resolution at a negative skein crossing}
\label{locSkeinTree2}
\end{figure}
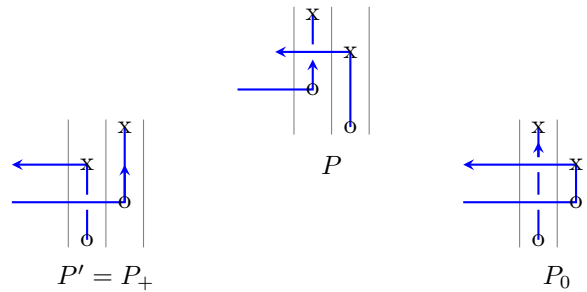
	
	\begin{proof}[Proof of Lemma \ref{lem:passingThru}]
	Firstly, $\lambda$ is not changed in passing from $P$ to either $P'$ or $P_0$. The columns depicted in Figure \ref{locSkeinTree} (resp.\ Figure \ref{locSkeinTree2}) are adjacent, so there is a $\beta$-curve, say $\beta_i$, between the two columns. The intersection of the projection with some $\beta$-curve determines $\lambda$, but the value of $\lambda$ is independent of the choice of $\beta$-curve for this intersection. Choose any $\beta$-curve other than $\beta_i$ (at least one exists, because the existence of a skein crossing implies at least two columns in the grid diagram). Since the horizontal arcs of the projections $P$, $P'$, and $P_0$ are identical outside of the two columns shown, $\lambda(P)=\lambda(P')=\lambda(P_0)$.
	
	Also, consider the pair of vertical arcs depicted in each projection $P$, $P'$, and $P_0$. The sum of the lengths of these two arcs is the same in all three projections. As $P$, $P'$ and $P_0$ are identical elsewhere, $\mu(P)=\mu(P')=\mu(P_0)$.
	
	Note that the number and nature of cusps in $P$ are the same as those in $P'$ and $P_0$, so $c_d(P)=c_d(P')=c_d(P_0)$. Therefore, the equality in (\ref{slIneq}) and the supposition that $-\slk_\rat(T_+(P))-i(P)<-\slk_\rat(T_+(P'))-i(P')$ imply that $\wt i(P)>\wt i(P')$. 
	
	However, by our assumption on the invariant $i$, we see that
	\al{
	&\wt i(L_+)\le\max\left(\wt i(L_-),\wt i(L_0)\right)\\
	\text{and}\\
	&\wt i(L_-)\le\max\left(\wt i(L_+),\wt i(L_0)\right).
	}
	
	Since $P'=P_-$ or $P'=P_+$ (depending on the sign of the skein crossing in question), $\wt i(P)>\wt i(P')$ implies that $\wt i(P)\le\wt i(P_0)$. But that implies that $-\slk_\rat(T_+(P))-i(P)\ge-\slk_\rat(T_+(P_0))-i(P_0)$, finishing the proof.
\end{proof}

Theorem \ref{mainThm2} has the following applications.

{\bf Corollary \ref{BennBd}}.\quad
{\it Let $L$ be an oriented link with some transverse representative $L_t$ in $(L(p,q),\xi_{UT})$. Let $J_{p,q}$ denote the HOMFLY polynomial invariant in $L(p,q)$, normalized as in Section \ref{subsec:HOMpoly}, and set $e(L)$ to be the minimum degree in $a$ of $J_{p,q}(L)$. Then
	\[\slk_\rat(L_t)\le\frac{e(L)-1}p.\]
}
\medskip
\begin{proof}
The defining skein relation of $J_{p,q}$ says that
	\[e(L_+)\ge\min\left(e(L_-)+2p,e(L_0)+p\right),\]
implying
\[-e(L_+)+1\le\max\left(-e(L_-)+1-2p,-e(L_0)+1-p\right).\]
Dividing by $p$ and then adding 1 we have
\[\frac{-e(L_+)+1}p+1\le\max\left(\frac{-e(L_-)+1}p-1,\frac{-e(L_0)+1}p\right).\]

A similar computation shows
	\[\frac{-e(L_-)+1}p-1\le\max\left(\frac{-e(L_+)+1}p+1,\frac{-e(L_0)+1}p\right).\]

Letting $i(L)=\frac{-e(L)+1}p$, we see that $i$ satisfies the first hypothesis of Theorem \ref{mainThm2}. Moreover, by our choice of normalization $e(\tau)=p\cdot\slk_\rat(T(\tau))+1$ for any trivial link $\tau$ and so $-i(\tau)=\slk_\rat(T(\tau))$. Since all conditions of Theorem \ref{mainThm2} are met, we are done.
\end{proof}

{\bf Corollary \ref{TrivmaxSL}}.\quad
{\it If $\tau$ is a trivial link in $L(p,q)$, then $T(\tau)$ has maximal self-linking number among all transverse representatives of $\tau$. If $\tau$ is a trivial knot, then the Legendrian knot associated to its grid number one diagram has maximal Thurston-Bennequin number.
}
\medskip

\begin{proof} Let $\tau_t$ be some transverse representative of $\tau$. Corollary \ref{BennBd} says that $\slk_\rat(\tau_t)\le\frac{e(\tau)-1}p=\slk_\rat(T(\tau))$.

Now consider a Legendrian knot $K$ that is associated to a grid number one diagram and let $K'$ be another Legendrian knot with the same knot type as $K$. Suppose that $\tb_\rat(K)<\tb_\rat(K')$. In \cite{BE} it is shown that $K$ and $K'$ are Legendrian isotopic after each has been positively and negatively stabilized some number of times, and therefore $\tb_\rat(K)-\tb_\rat(K')$ is an integer. So we have $\tb_\rat(K)+1\le\tb_\rat(K')$.

By Corollary \ref{rotForm}, if $P$ is a grid projection for a grid diagram of $K$, then 
\[\rot_\rat(K)=\frac12(c_d(P)-c_u(P))-\frac{\lambda-\mu}p.\]
Since we are in the grid number one case, $P$ has one vertical and one horizontal arc. Also we can choose $P$ so that $0<\mu(P)<p$ and $0<\lambda(P)<p$. This choice of $P$ has no cusps at all, implying that $\abs{\rot_\rat(K)}<1$.

By Corollary \ref{lem:slkTOtb-rot} the transverse pushoff $T(K)$ of $K$ has self-linking $\slk_\rat(T(K))=\tb_\rat(K)+\abs{\rot_\rat(K)}$. So since $\abs{\rot_\rat(K)}<1$, it must be that $\slk_\rat(T(K))<\tb_\rat(K')$. This contradicts the fact that $\slk_\rat(T(K))$ is maximal, since one of the positive or negative transverse pushoffs of $K'$ has self-linking number at least as large as $\tb_\rat(K')$.
\end{proof}

\section{Computations}
\label{sec:comput}

Consider the family of links $\set{L_n}_{n\ge0}$ where $L_n$ is the link associated to the grid diagram in Figure \ref{fig:L_n} with grid number $n+2$. The link $L_n$ is a knot if $n$ is odd and a 2-component link if $n$ is even. The first two columns of this grid diagram make a negative skein crossing. It is not difficult to see that for $n\ge2$, commutation of these columns of $L_n$ gives a link isotopic to $L_{n-2}$ and the resolution of the same columns gives a link isotopic to $L_{n-1}$. Therefore
	\[J_{5,1}(L_n) = a^{-10}J_{5,1}(L_{n-2})-a^{-5}zJ_{5,1}(L_{n-1}).\]

\begin{figure}[ht]
	\[\begin{tikzpicture}[scale=0.09]
	\draw[step=28,thick]
			(0,0) grid (140,12.5);
	\draw[step=28,thick,cm={1,0,0,1,(0,4)}]
			(0,11.5) grid (140,28);
	\draw[gray,thin,step=4]
			(0,0) grid (140,12.5)
			(0,15.5) grid (140,32);	
	\foreach \l in {0,1,2,3,4}
	\draw 
		(28*\l,0) node [below] {\l};
	\draw	
		(140,0) node [below] {0};	
	\foreach \l in {0,1,2,3,4}
	\draw
		(28*\l+28,32) node[above] {\l};
	\draw
		(0,32) node[above] {4};


	\draw[blue,thick]
		(22,32) -- (22,30) -- (114,30) -- (114,27) (114,25) -- (114,22) -- (121,22)
		(26,32) -- (26,31) (26,29) -- (26,26) -- (118,26) -- (118,23) (118,21) -- (118,18) -- (121,18)
		(127,10) -- (130,10) -- (130,7) (130,5) -- (130,2) -- (138,2) -- (138,0)
		(127,6) -- (134,6) -- (134,3) (134,1) -- (134,0);

	\foreach \x in {(114,30),(118,26),(130,10), (134,6), (138,2)}
		\draw 	\x node {x};
	\foreach \o in {(22,30),(26,26),(114,22),(118,18),(130,2)}
		\draw	\o node {o};
	
	\draw
		(122.5,32) node[above] {$\overbrace{\qquad\qquad}$}
		(122.5,35) node[above] {\small $n$ columns};
	\fill[white]
		(121,-2) -- (121,33) -- (127,33) -- (127,-2) -- cycle;
	\draw
		(124,14) node {$\ddots$};
	\filldraw
		(0,0) circle (16pt)
		(28,32) circle (16pt);
			
	\end{tikzpicture}\]
	\caption{Grid diagram associated to the link $L_n$.}
	\label{fig:L_n}
\end{figure}
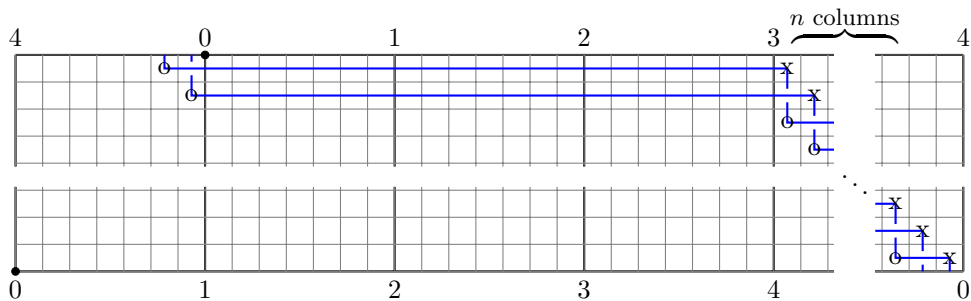

In \cite{C} we computed $J_{5,1}(L_0)=a^{-3}$ and $J_{5,1}(L_1)=a^{-8}(1-z)$. Define $f_n$ recursively: let
	\al{
		f_0= 1,\qquad	&\qquad f_1=1-z,
	}
	and define $f_n=f_{n-2}-zf_{n-1}$ for $n\ge2$. Then the skein relation above implies that $J_{5,1}(L_n) = a^{-5n-3}f_n$. Note that the recursive definition of $f_n$ implies that it is not zero for any $n$.

The grid diagram of Figure \ref{fig:L_n} determines a Legendrian link of $(L(5,1),\xi_{UT})$. Let $T(L_n)=T_+(L_n)$ denote positive transverse pushoff of $L_n$. By Corollary \ref{BennBd}
	\[\slk_\rat(T(L_n))\le\frac{e(L_n)-1}5 = -n-\frac45.\]
If we choose $P_n$ to be a grid projection of $L_n$, then by Corollary \ref{slForm}
	\[\slk_\rat(T(L_n)) = w(P_n)-c_d(P_n)-\frac{\mu\lambda+(\mu-\lambda)}5.\]
Let $P_n$ be the grid projection depicted in Figure \ref{fig:L_n}. Then $w(P_n)=-n-2$ and $c_d(P_n)=0$. Moreover, $\mu(P_n)=2$ and $\lambda(P_n)=-8$. Therefore,
	\[\slk_\rat(T(L_n)) = -n-2+\frac{6}5 = -n-\frac45,\]
showing that each $T(L_n)$ maximizes its self-linking number. We remark that 
	\[\tb_\rat(L_n)=-2n-\frac45\quad\text{ and }\quad\rot_\rat(L_n)=-n.\]

\bibliography{FWMInequalrefs}

\providecommand{\bysame}{\leavevmode\hbox to3em{\hrulefill}\thinspace}
\providecommand{\MR}{\relax\ifhmode\unskip\space\fi MR }
\providecommand{\MRhref}[2]{%
  \href{http://www.ams.org/mathscinet-getitem?mr=#1}{#2}
}
\providecommand{\href}[2]{#2}
\begin{thebibliography}{10}

\bibitem{BGH}
K.~L. Baker, J.~E. Grigsby, and M.~Hedden, \emph{Grid diagrams for lens spaces
  and combinatorial knot floer homology}, International Mathematics Research
  Notices \textbf{2008} (2008), 39 pages.

\bibitem{BE}
K.L. Baker and J.B. Etnyre, \emph{Rational linking and contact geometry}, To
  appear in Perspectives in Analysis, Geometry, and Topology (2009).

\bibitem{BG}
K.L. Baker and J.E. Grigsby, \emph{Grid diagrams and {L}egendrian lens space
  links}, J. Symp. Geom. \textbf{7} (2009), no.~4.

\bibitem{Bnn}
D.~Bennequin, \emph{Entrelacement et equations de {P}faff}, Asterisque
  \textbf{107--108} (1983), 83--161.

\bibitem{C}
C.~R. Cornwell, \emph{A polynomial invariant for links in lens spaces}, arXiv:
  math.GT/1002.1543v2.

\bibitem{E}
Y.~Eliashberg, \emph{Legendrian and transversal knots in tight contact
  3-manifolds}, Topological Methods in Modern Mathematics, 1993, pp.~171--193.

\bibitem{B}
J.~B. Etnyre, \emph{Legendrian and transversal knots}, Handbook of knot theory,
  Elsevier B. V., 2005, pp.~105--185.

\bibitem{FS}
R.~Fintushel and R.~J. Stern, \emph{Constructing lens spaces by surgery on
  knots}, Math. Z. \textbf{175} (1980), 33--51.

\bibitem{FW}
J.~Franks and R.~Williams, \emph{Braids and the {J}ones polynomial}, Trans. AMS
  \textbf{303} (1987), 97--108.

\bibitem{FT}
D.~Fuchs and S.~Tabachnikov, \emph{Invariants of {L}egendrian and transverse
  knots in the standard contact space}, Topology \textbf{36} (1997), no.~5,
  1025--1053.

\bibitem{Ge}
H.~Geiges, \emph{Contact geometry}, Handbook of differential geometry. {V}ol.
  {II}, Elsevier/North Holland, Amsterdam, 2006, pp.~315 -- 382.

\bibitem{H}
M.~Hedden, \emph{An {O}zsva\'th-{S}zabo\' {F}loer {H}omology invariant of knots
  in a contact manifold}, math.GT/0708.0448v2.

\bibitem{KL}
E.~Kalfagianni and X.-S. Lin, \emph{The {HOMFLY} polynomial for links in
  rational homology 3-spheres}, Topology \textbf{38} (1999), no.~1, 95--115.

\bibitem{LM}
P.~Lisca and G.~Matic, \emph{Stein 4-manifolds with boundary and contact
  structures. \emph{from} symplectic, contact, and low-dimensional topology},
  Topol. Appl. \textbf{88} (1998), no.~1-2, 55--66.

\bibitem{M}
H.R. Morton, \emph{Seifert circles and knot polynomials}, Math. Proc. Camb.
  Phil. Soc. \textbf{99} (1986), 107--109.

\bibitem{MR}
T.~Mrowka and Y.~Rollin, \emph{Legendrian knots and monopoles}, Algebraic and
  Geometric Topology \textbf{6} (2006), 1--69.

\bibitem{Ng1}
L.~Ng, \emph{A {L}egendrian {T}hurston--{B}ennequin bound from {K}hovanov
  homology}, Algebr. Geom. Topol. \textbf{5} (2005), 1637--1653,
  math.GT/0508649.

\bibitem{Ng}
\bysame, \emph{A skein approach to {B}ennequin type inequalities}, Int. Math.
  Res. Not. (2008), 18 pages, math.GT/0709.2141.

\bibitem{NT}
L.~Ng and D.~Thurston, \emph{Grid diagrams, braids, and contact geometry},
  2008, pp.~120--136.

\bibitem{OS}
P.~Ozsv\'ath and Z.~Szab\'o, \emph{Heegaard {F}loer homology and contact
  structures}, Duke Math. J. \textbf{129} (2005), no.~1, 39--61.

\bibitem{OST}
P.~Ozsvath, Z.~Szabo, and D.~Thurston, \emph{Legendrian knots, transverse knots
  and combinatorial {F}loer homology}, Geom. Topol. \textbf{12} (2008),
  941--980.

\bibitem{Pla1}
O.~Plamenevskaya, \emph{Bounds for the {T}hurston--{B}ennequin number from
  {F}loer homology}, Algebr. Geom. Topol. \textbf{4} (2004), 399--406,
  math.GT/0311090.

\bibitem{Pla2}
\bysame, \emph{Transverse knots and {K}hovanov homology}, Math. Res. Lett.
  \textbf{13} (2006), no.~4, 571--586, math.GT/0412184.

\bibitem{Rud1}
L.~Rudolph, \emph{A congruence between link polynomials}, Math. Proc. Cambridge
  Philos. Soc. \textbf{107} (1990), 319--327.

\bibitem{Rud2}
\bysame, \emph{Quasipositivity as an obstruction to sliceness}, Bull. Amer.
  Math. Soc. (N.S.) \textbf{29} (1993), 51--59.

\bibitem{Rud3}
\bysame, \emph{An obstruction to sliceness via contact geometry and
  ``classical'' gauge theory}, Invent. Math. \textbf{119} (1995), 155--163.

\bibitem{Ruth}
D.~Rutherford, \emph{The {B}ennequin number, {K}auffman polynomial, and ruling
  invariants of a {L}egendrian link: the {F}uchs conjecture and beyond}, Int.
  Math. Res. Not. (2006).

\bibitem{Shum}
A.~Shumakovitch, \emph{Rasmussen invariant, slice-{B}ennequin inequality, and
  sliceness of knots}, JKTR \textbf{16} (2007), no.~10, 1403--1412.

\bibitem{Wu2}
H.~Wu, \emph{On the slicing genus of {L}egendrian knots}, arXiv:
  math.GT/0505279v2.

\bibitem{Wu}
\bysame, \emph{Braids, transversal links and the {K}hovanov--{R}ozansky
  {T}heory}, Trans. Amer. Math. Soc. \textbf{360} (2008), 3365--3389.

\end{thebibliography}
\bibliographystyle{amsplain}
	
\end{document}